\documentclass[a4paper,11pt]{article}
\title{Non-degeneracy and Fibering with Applications to Diophantine Approximation}
\author{ZiAn Cody ZHAO}
\date{\today} 
\usepackage{graphicx}
\usepackage{float}
\usepackage{dsfont}
\usepackage{framed}
\usepackage{amsmath}
\usepackage{amssymb}
\usepackage{amsfonts}
\usepackage{mathrsfs}
\usepackage{array}
\usepackage{hyperref}
\usepackage{hologo} 
\usepackage{tikz}

\usepackage{graphicx}
\allowdisplaybreaks
\usepackage[numbers]{natbib}
\usepackage{cite}

\usepackage{url}

\usepackage{amsthm}



\theoremstyle{plain} 
\newtheorem{thm}{Theorem}[section] 
\newtheorem{lem}[thm]{Lemma} 
\newtheorem{prop}[thm]{Proposition}

\theoremstyle{definition}

\theoremstyle{remark} 
\newtheorem*{rem}{Remark}

\begin{document}
\maketitle 
\begin{abstract}
Motivated by applications in Diophantine approximation, we develop a new fibering mechanism for non-degenerate manifolds in $\mathbb{R}^n$, showing that they admit local decompositions into families of non-degenerate curves with uniform quantitative non-degeneracy properties. The manifolds are required to be $C^N$ with the essentially optimal bound $N\geq \max\lbrace n,l+1\rbrace$, where the manifold is $l$ non-degenerate. Furthermore, we show that this uniform quantitative local control can be globalised into a qualitative statement, so that almost every point of a non-degenerate manifold lies on a non-degenerate curve contained in the manifold. Our approach is perturbative and shows that the fibering structure arising in the analytic setting is stable under sufficiently small $C^N$ perturbations. The method provides a framework for extending fibering based results in Diophantine approximation from curves to manifolds. Such applications are considered in the paper.
\end{abstract}
\tableofcontents
\section{Introduction}
One of the central themes in Diophantine approximation is to determine the Hausdorff dimension and measure of sets of points on manifolds satisfying prescribed Diophantine properties \citep{Beresnevich_2015,beresnevich2020winningpropertybadlyapproximable,beresnevich2026twisteddiophantineapproximationmanifolds,chow2026rationalpointshyperbolicregions,beresnevich2026weighteddiophantineapproximationmanifolds}. A common strategy is to reduce the problem from manifolds to curves via a fibering construction. Once the desired statement has been established for curves, one can lift it back to the original manifold using geometric arguments such as Fubini's theorem or Marstrand’s slicing lemma \citep[Theorem 10.11]{Mattila_1995}.

The scope of this approach is largely determined by the available fibering theory. Existing fibering constructions require the manifold to be either analytic, see \citep[pp.~9--10]{Sprindzhuk1980ACHIEVEMENTSAP} and \citep[Appendix C]{Beresnevich_2015}, or differentiable to a relatively high order, see \citep[Lemma A.1]{chow2026rationalpointshyperbolicregions}. In contrast, our theorem establishes a new fibering mechanism under an essentially optimal smoothness requirement. At the local level, we obtain uniform quantitative control of non-degeneracy, providing a lower bound for the Wronskian that holds simultaneously for all fibered curves and all sufficiently small $C^N$ perturbations, making the construction particularly suited to dimension theoretic applications. This uniform quantitative control can then be globalised into a qualitative statement showing that almost every point of a non-degenerate manifold lies on a non-degenerate curve, providing the corresponding fibering mechanism for measure theoretic applications. The effective nature of this quantitative control may also provide a form of geometric rigidity under perturbation needed to develop higher-dimensional analogues of weak non-degeneracy \citep{BERESNEVICH20103064}.

In \citep{Pyartli1969}, Pyartli formulated non-degeneracy through fibering into non-degenerate curves; our result shows that every $l$ non-degenerate manifold in $\mathbb{R}^n$, with $C^{\max\{n,l+1\}}$ regularity, admits such a fibering.

The paper is organised as follows. In Section~1, we introduce the necessary definitions and state our main results. In Section~2, we establish our new fibering results for non-degenerate manifolds. In Section~3, we apply these results to badly approximable points and multiplicative Diophantine approximation on manifolds, and establish new results in twisted Diophantine approximation on manifolds. 
\subsection{Notation}
For a multi-index \(\alpha=(\alpha_1,\ldots,\alpha_m)\in\mathbb Z_{\geq0}^m\) and \(x=(x_1,\ldots,x_m)\in\mathbb R^m\), let
\[
|\alpha|
=
\alpha_1+\cdots+\alpha_m,
\quad
x^\alpha
:=
x_1^{\alpha_1}\cdots x_m^{\alpha_m},
\]
and define
\[
\partial_\alpha
:=
\frac{\partial^{|\alpha|}}
{\partial x_1^{\alpha_1}\cdots\partial x_m^{\alpha_m}}.
\]
The symbol \( |\cdot| \) denotes the absolute value of a real number, the Euclidean norm of a vector, and the length of a multi-index; the intended meaning will always be clear from the context. Let \(B\subset\mathbb R^m\) be a ball and denote its closure by \(\overline B\). For a map \(f:\overline B\to\mathbb R^n\), we write
\[
\|f\|_{C^N(\overline B)}
=
\max_{|\alpha|\le N}
\sup_{x\in\overline B}
\vert\partial_\alpha f(x)\vert.
\]
The symbol \(\mathbb{R}^\times\) denotes the set of non-zero real numbers, so that \((\mathbb{R}^\times)^m\) consists of vectors in \(\mathbb{R}^m\) with all coordinates non-zero. 

For $C^n$ functions $f_1,\dots,f_n:I\to\mathbb R$, define the Wronskian by
\[
W(f_1',\dots,f_n')(t)
=
\det\!
\begin{pmatrix}
f_1'(t) & \cdots & f_n'(t)\\
\vdots & \ddots & \vdots\\
f_1^{(n)}(t) & \cdots & f_n^{(n)}(t)
\end{pmatrix}.
\]
\subsection{Non-degeneracy and analytic fibering}
The notion of non-degeneracy, introduced in \citep{kleinbock1998flowshomogeneousspacesdiophantine} for $C^l$ smooth manifolds $\mathcal{M}\subset\mathbb{R}^n$, for some $l\in\mathbb{N}$, is essentially that the manifold is sufficiently curved so as not to lie locally in any affine hyperplane. Formally, an $m$-dimensional manifold $\mathcal{M}$ embedded in $\mathbb{R}^n$ is said to be non-degenerate if almost every point $x_0$ of $\mathcal{M}$ admits a neighbourhood that is parametrised by an $l$ non-degenerate map $f:U\subset\mathbb{R}^m\to\mathbb{R}^n$ for some $l\in\mathbb{N}$, where $U$ is an open subset of $\mathbb{R}^m$ and $x_0\in f(U)$. The map $f:U\to\mathbb{R}^n$, $x\mapsto f(x)=(f_1(x),\dots,f_n(x))$, is said to be $l$ non-degenerate at $x\in U$ if $f$ is $l$ times continuously differentiable on some sufficiently small ball centred at $x$ and
\begin{align}\label{2}
\operatorname{span}\{\partial_\alpha f(x):\,1\le|\alpha|\le l\}=\mathbb{R}^n.
\end{align}
The map $f$ is said to be non-degenerate if it is non-degenerate at almost every point in $U$ with respect to $m$-dimensional Lebesgue measure. 

In the case of curves, where $m=1$, a map $f:I\subset\mathbb{R}\to\mathbb{R}^n$ is said to be Wronskian non-degenerate at $x_0\in I$ if $f$ is $n$ non-degenerate at a point $x_0$. Equivalently,
\[
W(f_1',\dots,f_n')(x_0)\neq 0.
\]

For curves, Wronskian non-degeneracy is stronger than $l$ non-degeneracy, as it requires the first $n$ derivatives to span $\mathbb{R}^n$, without the need for higher-order derivatives. This is one of the features that makes the main fibering result of this paper particularly noteworthy, as it provides a mechanism for fibering an $l$ non-degenerate manifold into Wronskian non-degenerate curves. In this sense, our fibering theorem establishes a bridge between these two notions of non-degeneracy.

On the other hand, Wronskian non-degeneracy is not only a cleaner notion for curves, but also the one appearing in several fundamental results in Diophantine approximation, such as the theorem of Bernik, Kleinbock and Margulis \citep[Theorem 1.4]{bernik2001khintchine}. Moreover, by \citep[Proposition 2.13]{beresnevich2023khintchine}, if $f$ is $l$ non-degenerate at $x_0$, then $f$ is Wronskian non-degenerate outside a countable set in some neighbourhood of $x_0$. Consequently, although $l$ non-degeneracy is the natural notion for manifolds, Wronskian non-degeneracy reflects the generic behaviour of curves. This is one reason why our main fibering theorem is formulated in terms of Wronskian non-degenerate fibers.

Both Wronskian non-degeneracy (for curves) and $l$ non-degeneracy (for general manifolds) imply that the functions $1,f_1,\dots,f_n$ are linearly independent. The converse, however, need not hold in general. In the analytic setting, these notions are closely related: for analytic maps on a connected domain, the linear independence of $1,f_1,\dots,f_n$ is equivalent to finite order non-degeneracy at every point, and, for analytic curves, to the Wronskian not vanishing identically; see, for example, \citep[Theorem 1]{Bostan_2010} and \citep[pp.~91--92]{4ade5a02-1dbf-3a2a-872f-9a5afb3d63a1}. Since these results are well known, we omit the proofs.

We now recall the classical analytic fibering lemma, which provides a mechanism for reducing the non-degeneracy of analytic maps to that of curves; see \citep[pp.~9--10]{Sprindzhuk1980ACHIEVEMENTSAP} and \citep[Appendix C]{Beresnevich_2015}.
\begin{lem}[Fibering for analytic manifolds]\label{lem2}
Let $f_0,\dots,f_n$ be analytic functions in $m$ real variables defined on an open neighbourhood of $0$. Assume that $f_0,\dots,f_n$ are linearly independent over $\mathbb{R}$. Then there is a sufficiently large integer $d_0>1$ such that for every $d>d_0$ and every $u=(u_1,u_2,\dots,u_m)\in \mathbb{R}^m$ with $u_1\cdots u_m\neq 0$, the following functions of real variable
\[
\phi_{u,i}:E_u\mapsto \mathbb{R}\quad (0\leq i\leq n)
\]
given by
\[
\phi_{u,i}(t):=f_i(u_1t^{1+d^m},u_2t^{d+d^m},\dots,u_mt^{d^{m-1}+d^m}),
\]
where $E_u\subset \mathbb{R}$ is a neighbourhood of $0$, are linearly independent over $\mathbb{R}$.
\end{lem}
This naturally raises the question of whether there exists a more general fibering mechanism capable of reducing finite smoothness non-degenerate manifolds to Wronskian non-degenerate curves, and what the minimum smoothness requirement for such a mechanism should be. The remainder of this paper is devoted to developing precisely such a fibering mechanism.
\subsection{Main results}
The key idea is to exploit the stability of suitable non-degeneracy data under sufficiently small perturbations. Starting from a $C^N$ $l$ non-degenerate manifold, we first identify a finite collection of pivot derivatives encoding its non-degeneracy. Since these derivatives vary continuously under $C^N$ perturbations, the corresponding pivot determinant remains uniformly bounded away from zero throughout a sufficiently small neighbourhood of the original map. This pivot structure then provides the quantitative information required to obtain uniform lower bounds for the Wronskian determinants of the associated fibered curves.

The main difficulty is that the analytic fibering lemma \ref{lem2} provides only a qualitative non-degeneracy statement. Although each resulting fibered curve has non-vanishing Wronskian determinant, the value of the Wronskian cannot be controlled as the perturbation varies. Consequently, this information alone is insufficient to pass to the limit as the error of perturbation tends to zero. A genuinely quantitative refinement of the analytic fibering mechanism is therefore required.

Once such uniform control is available, the density of analytic maps in the $C^N$ topology allows us to approximate the original smooth manifold by analytic perturbations, while the continuity of the Wronskian determinant ensures that the corresponding non-degeneracy properties are preserved in the limit. This new local fibering mechanism is now formalized as the following theorem.
\begin{thm}[Local fibering of smooth manifolds]\label{thmfib}
Let $B\subset\mathbb{R}^m$ be a ball centred at $0$, and let $f:\overline{B}\to \mathbb{R}^n$ be a $C^N$ map. Suppose there exist non-zero multi-indices $\alpha_1,\dots,\alpha_n$ such that
\begin{align}\label{46}
\vert \det (\partial_{\alpha_1}f(0),\dots,\partial_{\alpha_n}f(0))\vert>0.
\end{align}
Set
\begin{align*}
l=\max_{1\leq i\leq n}\vert \alpha_i\vert,
\end{align*}
and suppose
\begin{align}\label{45}
N\geq\max\lbrace n,l+1\rbrace.
\end{align}
Then there exists $d_0\in \mathbb{N}$ satisfying the following. For every compact set $U_0\subset (\mathbb{R}^\times)^m$ and every $d\geq d_0$, define
\begin{align}\label{5}
a_j=d^{j-1}+d^m,\quad \gamma_u(t)=(u_1t^{a_1},\dots,u_mt^{a_m}).
\end{align} 
There exist $c_*>0$ and \(0<t_*\leq1\) such that for every \(u\in U_0\), $\gamma_u([0,t_*])\subset B$ and
\[
\vert W((f_{1}\circ\gamma_u)',\dots,(f_{n}\circ\gamma_u)')(t_u)\vert \geq c_*,
\]
for some $t_u\in [0,t_*]$.
\end{thm} 
\begin{rem}
By continuity of the Wronskian, for each $u\in U_0$ there exists an
interval $I_u\subset[0,t_*]$ containing $t_u$ such that
\[
\left\vert
W((f_1\circ\gamma_u)',\dots,(f_n\circ\gamma_u)')(t)
\right\vert
\geq \frac{c_*}{2}
\]
for all $t\in I_u$. Hence the curve
\[
\phi_u(t):=f(\gamma_u(t))
\]
is Wronskian non-degenerate at every point of $I_u$. In particular, one may choose
$s_u\in I_u\cap(0,t_*)$ at which $\phi_u$ is Wronskian non-degenerate.
\end{rem}
\begin{rem}
The smoothness requirement differs from the natural lower bound $\max\{n,l\}$ by only one derivative. Indeed, the condition $N\ge n$ is unavoidable since the Wronskian involves derivatives up to order $n$, while $N\ge l$ is necessary for the notion of $l$ non-degeneracy itself. Moreover, within the present fibering framework, the additional derivative cannot in general be removed. In this sense, the smoothness requirement is essentially optimal.
\end{rem}
\begin{rem}
Suppose, in addition, that \(f\) is a local parametrisation of an
\(m\)-dimensional manifold. Then \(Df(0)\) has rank \(m\). Consequently,
\(m\) independent directions of non-degeneracy are already supplied by
the first derivatives, so only the remaining \(n-m\) directions need to
be witnessed by higher-order derivatives. This observation immediately yields several simplifications of the smoothness requirement \eqref{45}. In particular, if \(l=2\) and \(m<n\), then $m\geq 2$ implies $n\ge3$, and hence
\[
\max\{n,l+1\}
=
\max\{n,3\}
=
n.
\]
Thus the smoothness requirement reduces to $N\ge n$. Furthermore, in the hypersurface case $n-m=1$, the smoothness requirement can be further reduced to $N\ge\max\{n,l\}$; see Proposition~\ref{prop:hypersurface}.
\end{rem}
There remains, however, a distinction between local and global fibering. Theorem~\ref{thmfib} guarantees a Wronskian non-degenerate segment along each fiber, but not that this segment contains the point from which the construction is initiated. We overcome this by exploiting the quantitative uniformity of the construction: at every sufficiently small scale, the Wronskian non-degenerate portions of the fibers collectively occupy a fixed positive proportion of a suitable neighbourhood. Such regions cannot repeatedly avoid a set of positive measure, since at a density point such a set occupies asymptotically full measure in sufficiently small neighbourhoods. This leads to the following global fibering theorem.
\begin{thm}[Global fibering of smooth manifolds]\label{thm:globalfib}
Let $V\subset\mathbb{R}^m$ be a bounded non-empty open set and let
$f:V\to\mathbb{R}^n$ be a $C^N$ parametrisation of an
$m$-dimensional manifold $\mathcal{M}=f(V)$. Suppose that $f$ is
$l$ non-degenerate at almost every point of $V$ and that
\[
N\geq\max\{n,l+1\}.
\]
Then, for almost every $x\in\mathcal{M}$ with respect to $m$-dimensional Hausdorff measure on $\mathcal{M}$, there exist an interval
$I\subset\mathbb{R}$, a $C^N$ curve
\[
\phi:I\to\mathcal{M},
\]
and $t_x\in I$ such that
\[
\phi(t_x)=x
\quad
\text{and}
\quad
W(\phi_1',\dots,\phi_n')(t_x)\neq0.
\]
In particular, almost every point of $\mathcal{M}$ lies on a
Wronskian non-degenerate curve contained in $\mathcal{M}$.
\end{thm}
\begin{rem}
Theorem~\ref{thm:globalfib} in the above form is not immediately
applicable as written, since it does not provide a parametrised family of non-degenerate curves.
For the sake of a clean statement, however, we state the result in this form. There is nevertheless a quite clean way to use the construction underlying the proof of Theorem~\ref{thm:globalfib} and combine it with a density argument to obtain measure theoretic statements; see the proof of Theorem~\ref{thmcgp}.
\end{rem}

The quantitative control and perturbative stability described at the beginning of this section are established in the following theorem, which constitutes the main technical result of this paper and implies Theorem \ref{thmfib}.
\begin{thm}[Uniform Perturbative Non-degenerate Fibering]\label{upnf}
Let $B\subset \mathbb{R}^m$ be a ball centred at $0$, and let
$f:\overline{B}\to \mathbb{R}^n$ be a $C^N$ map. Suppose there exist
non-zero multi-indices $\alpha_1,\dots,\alpha_n$ such that \eqref{46} holds. Let $l$ and $N$ be as in Theorem \ref{thmfib}. Then there exists $d_0\in\mathbb N$ such that, for every compact set $U_0\subset (\mathbb{R}^\times)^m$ and every \(d\ge d_0\), there exist $\varepsilon_1>0$, \(0<t_*\le1\), and \(c_*>0\) for which the following holds. Let $\gamma_u$ be defined by \eqref{5}; then $\gamma_u([0,t_*])\subset B$ for every $u\in U_0$. Let
$0<\varepsilon\leq\varepsilon_1$, and let
$f_\varepsilon:\overline{B}\to\mathbb{R}^n$
be an analytic map satisfying
\[
\|f_\varepsilon-f\|_{C^N(\overline{B})}<\varepsilon.
\]
Then, for every $(\varepsilon,u)\in(0,\varepsilon_1]\times U_0$, there exists
\[
t_{\varepsilon,u}\in[0,t_*]
\]
such that
\[
\bigl|
W((f_{\varepsilon,1}\circ\gamma_u)',
\dots,
(f_{\varepsilon,n}\circ\gamma_u)')
(t_{\varepsilon,u})
\bigr|
\geq 2c_*,
\]
uniformly for $0<\varepsilon\leq\varepsilon_1$ and
$u\in U_0$. Consequently, for every $(\varepsilon,u)$, there exists an interval $I_{\varepsilon,u}\subset [0,t_*]$ containing $t_{\varepsilon,u}$ such that
\[
\bigl|
W((f_{\varepsilon,1}\circ\gamma_u)',
\dots,
(f_{\varepsilon,n}\circ\gamma_u)')
(t)
\bigr|
\geq c_*
\]
for all $t\in I_{\varepsilon,u}$, uniformly for
$0<\varepsilon\leq\varepsilon_1$ and $u\in U_0$.
\end{thm}
\begin{rem}
In the proof, we may shrink $t_*$, if necessary, so that
\[
\gamma_u([0,t_*])\subset B
\]
for every \(u\in U_0\); this is possible since $U_0$ is compact so that $\vert \gamma_u(t)\vert\leq C_{U_0}t^{a_1}$ by \eqref{28}. The threshold $\varepsilon_1$ is chosen only after $U_0$, $d$ and $t_*$ have been fixed, and in general depends on them.
\end{rem}
We stress that the analyticity of the perturbation is not essential. It is introduced solely to facilitate the approximation argument leading to Theorem \(\ref{thmfib}\). Once the uniform quantitative lower bound has been obtained, the continuity of the Wronskian determinant in the $C^n$ topology implies that the conclusion of Theorem \(\ref{thmfib}\) is stable under arbitrary sufficiently small $C^n$ perturbations of $f$, not necessarily analytic. 
\section{Fibering of smooth non-degenerate manifolds}
\subsection{Local fibering}
The goal of this section is to develop the framework required to prove Theorem \ref{upnf}. Throughout this section, we assume the same hypotheses as in Theorem \ref{upnf}. The proof naturally splits into three main parts. In the first part, we construct suitable pivot data encoding the non-degeneracy of the original map and show that these data persist throughout a sufficiently small $C^N$ neighbourhood. In the second part, we use Taylor expansions to decompose the fibered maps into a polynomial term and a remainder term. We then analyse the polynomial contribution to the Wronskian determinant through the pivot construction, thereby obtaining a uniform quantitative lower bound. In the third part, we exploit the finite smoothness of the original manifold to show that the contribution of the polynomial term dominates that of the remainder. Combining these ingredients yields the proof of Theorem \ref{upnf}. After establishing Theorem \ref{upnf}, we show how it immediately implies the smooth fibering result stated in Theorem \ref{thmfib}.

We begin by recalling the following technical lemma from
\citep[Lemma 12]{Beresnevich_2015}, which introduces the exponent map underlying the fibering parametrisation.
\begin{lem}\label{lem4}
Let $d_0\geq 1$ and $d\geq \max\lbrace d_0,2\rbrace$ be integers and let $e_d:\mathbb{Z}_{\geq 0}^m\to\mathbb{Z}_{\geq 0}$ be given by
\begin{align}\label{12}
e_d(\alpha_1,\dots,\alpha_m):=\sum_{j=1}^m \alpha_j(d^{j-1}+d^m).
\end{align}
Let
\[
S_{d_0}:=\lbrace (\alpha_1,\dots,\alpha_m)\in \mathbb{Z}_{\geq 0}^m:\,\alpha_1+\cdots+\alpha_m\leq d_0\rbrace.
\]
Then
\begin{itemize}
\item[(i)] $e_d$ maps $S_{d_0}$ into $\mathbb{Z}_{\geq 0}$ injectively, and
\item[(ii)] $e_d(S_{d_0})\cap e_d(\mathbb{Z}_{\geq 0}^m\setminus S_{d_0})=\emptyset$.
\end{itemize}
\end{lem}
For each multi-index \(\alpha\in\mathbb Z_{\ge0}^m\), define
\[
\Lambda_\alpha
:=
\left(
\frac{\partial_\alpha f_0(0)}{\alpha!},\frac{\partial_\alpha f_1(0)}{\alpha!},
\dots,
\frac{\partial_\alpha f_n(0)}{\alpha!}
\right)
\in\mathbb R^{n+1},
\quad
f_0\equiv1.
\]
For an analytic perturbation \(f_\varepsilon\), define
\(\Lambda_\alpha(\varepsilon)\) analogously, with
\(f_{\varepsilon,0}\equiv1\).
\subsubsection{Construction and persistence of pivot data}
The first step is to replace the original non-degeneracy data by a distinguished collection of pivot derivatives. These satisfy a minimality property and will serve as the quantitative non-degeneracy information throughout the remainder of the proof.
\begin{lem}[Pivot construction]\label{lemp}
Let \(B\subset\mathbb R^m\) be a ball centred at \(0\), and let \(f:\overline B\to\mathbb R^n\) be a \(C^N\) map. Assume there exist non-zero multi-indices \(\alpha_1,\dots,\alpha_n\) such that \eqref{46} holds. Let \(l\) and \(N\) be as in Theorem~\ref{thmfib}. Then, with $d_0:=\max\lbrace N,l+1\rbrace$, there exist multi-indices
\[
\beta_0=0,\beta_1,\dots,\beta_n,
\qquad |\beta_i|\le l,
\]
such that, for every \(d\ge d_0\),
\begin{align}\label{21}
0=p_0<p_1<\cdots<p_n,
\quad
p_i=e_d(\beta_i).
\end{align}
Moreover, there exists \(c_{\mathrm{piv}}>0\) such that
\[
\left|
\det\bigl(
\partial_{\beta_1}f(0),
\dots,
\partial_{\beta_n}f(0)
\bigr)
\right|
\geq c_{\mathrm{piv}}.
\]
In particular,
\[
p_n\leq
\max_{1\leq j\leq n}e_d(\alpha_j).
\]
These pivot exponents are minimal in the following sense: if $\eta_0',\dots,\eta_n'$ are multi-indices with $\vert \eta_i'\vert\leq l$ and $\det\bigl(
\Lambda_{\eta_0'},
\dots,
\Lambda_{\eta_n'}
\bigr)\neq0$, whose exponents $q_i':=e_d(\eta_i')$ satisfy
\[
q_0'<q_1'<\cdots<q_n',
\]
then 
\[
p_i\leq q_i'
 \quad (0\leq i\leq n).
\]
\end{lem}
\begin{proof}
Since \(f_0\equiv 1\), we have
\[
\Lambda_0=(1,f_1(0),\dots,f_n(0)),
\]
while, for every non-zero multi-index \(\alpha\),
\[
\Lambda_\alpha
=
\left(
0,
\frac{\partial_\alpha f_1(0)}{\alpha!},
\dots,
\frac{\partial_\alpha f_n(0)}{\alpha!}
\right).
\]
Therefore, by \eqref{46},
\begin{align*}
\left|
\det(
\Lambda_0,
\Lambda_{\alpha_1},
\dots,
\Lambda_{\alpha_n})
\right|
=
\frac{
\left|
\det(
\partial_{\alpha_1}f(0),
\dots,
\partial_{\alpha_n}f(0))
\right|
}{
\alpha_1!\cdots\alpha_n!
}>0.
\end{align*}
Hence $\Lambda_0,\Lambda_{\alpha_1},\dots,\Lambda_{\alpha_n}$ span \(\mathbb R^{n+1}\). Since
\(|\alpha_i|\leq l\) for every \(1\leq i\leq n\), this spanning
collection is contained in $\{\Lambda_\alpha:|\alpha|\leq l\}$. By Lemma \ref{lem4}, \(e_d\) is injective on
\[
S_l=\{\alpha\in\mathbb Z_{\geq0}^m:|\alpha|\leq l\}
\]
whenever \(d\geq \max\lbrace l,2\rbrace\), and in particular whenever $d\geq d_0$, since $d_0=\max\lbrace N,l+1\rbrace\geq 2$. Moreover, the relative ordering of the values $\{e_d(\alpha):\alpha\in S_l\}$ is independent of \(d\geq d_0\). Indeed, for distinct \(\alpha,\alpha'\in S_l\),
\[
e_d(\alpha)-e_d(\alpha')
=
(|\alpha|-|\alpha'|)d^m
+
\sum_{j=1}^m(\alpha_j-\alpha_j')d^{j-1}.
\]
If $\vert \alpha\vert <\vert \alpha'\vert$, then, since $d\geq d_0>l$,
\[
e_d(\alpha)\leq \vert\alpha\vert d^m+ld^{m-1}<(\vert \alpha\vert+1)d^m\leq e_d(\alpha').
\]
If $\vert \alpha\vert =\vert \alpha'\vert$, let $k$ be the largest index with $\alpha_k\neq \alpha_k'$. Since $\vert \alpha_j-\alpha_j'\vert\leq l\leq d-1$ for every $j$, we have $\vert \sum_{j<k}(\alpha_j-\alpha_j')d^{j-1}\vert\leq d^{k-1}-1$, so the sign of $e_d(\alpha)-e_d(\alpha')$ is that of $\alpha_k-\alpha_k'$, which does not depend on $d$. Fix this ordering and write
\[
0=q_0<q_1<\cdots<q_r
\]
for the values \(e_d(\alpha)\) with \(|\alpha|\leq l\). By injectivity, for each \(0\leq j\leq r\) there is a unique multi-index
\(\eta_j\in S_l\) such that $q_j=e_d(\eta_j)$. The multi-indices \(\eta_j\) are therefore independent of the particular choice of \(d\ge d_0\).

We now construct the pivots recursively. Since \(q_0=0\), necessarily
\(\eta_0=0\), and $\Lambda_{\eta_0}=\Lambda_0\neq0$. Set $j_0=0$. Suppose that $0=j_0<j_1<\cdots<j_{k-1}$ have been chosen so that
\[
\Lambda_{\eta_{j_0}},
\dots,
\Lambda_{\eta_{j_{k-1}}}
\]
are linearly independent. Define \(j_k\) to be the smallest index
\(j>j_{k-1}\) such that
\[
\Lambda_{\eta_j}
\notin
\operatorname{span}
\{
\Lambda_{\eta_{j_0}},
\dots,
\Lambda_{\eta_{j_{k-1}}}
\}.
\]
Such an index exists until \(k=n\), because $\Lambda_0,\Lambda_{\alpha_1},\dots,\Lambda_{\alpha_n}$ already span \(\mathbb R^{n+1}\). Thus the construction terminates with
indices
\[
0=j_0<j_1<\cdots<j_n\leq r.
\]
Set
\[
\beta_k:=\eta_{j_k},
\quad
p_k:=e_d(\beta_k)=q_{j_k},
\quad
0\leq k\leq n.
\]
Then
\[
\beta_0=0,
\quad
|\beta_k|\leq l,
\]
and
\[
0=p_0<p_1<\cdots<p_n.
\]
Furthermore, $\Lambda_{\beta_0},\dots,\Lambda_{\beta_n}$ are linearly independent, and therefore
\[
\left|
\det(
\Lambda_{\beta_0},
\dots,
\Lambda_{\beta_n})
\right|=\frac{
\left|
\det(
\partial_{\beta_1}f(0),
\dots,
\partial_{\beta_n}f(0))
\right|
}{
\beta_1!\cdots\beta_n!
}>0.
\]
Since the multi-indices \(\beta_0,\dots,\beta_n\) are fixed, we may
choose a constant \(c_{\mathrm{piv}}>0\) such that
\[
\left|
\det(
\partial_{\beta_1}f(0),
\dots,
\partial_{\beta_n}f(0))
\right|
\ge c_{\mathrm{piv}}.
\]
The minimality property is immediate from the greedy construction, since \(p_k\) is chosen to be the smallest exponent at which the span increases from dimension \(k\) to \(k+1\). The bound
\[
p_n\leq \max_{1\leq j\leq n}e_d(\alpha_j)
\]
follows immediately by applying the minimality property to the spanning collection $\Lambda_0,\Lambda_{\alpha_1},\dots,\Lambda_{\alpha_n}$. This completes the proof.
\end{proof}
The next step is to show that the pivot data constructed above persist throughout a sufficiently small \(C^N\) neighbourhood of the original map.
\begin{lem}[Persistence of pivot data]\label{lem3}
Let the notation and assumptions be as in Lemma~\ref{lemp}, and let \(M=\|f\|_{C^l(\overline B)}+1\). Then there exists
\[
0<\varepsilon_0<
\min\left\{
1,
\frac{c_{\mathrm{piv}}}
{2nn!M^n}
\right\}
\]
such that the following holds. If \(f_\varepsilon\) is analytic and satisfies
\begin{align}\label{7}
\|f_\varepsilon-f\|_{C^N(\overline B)}
<\varepsilon,
\quad
0<\varepsilon\le\varepsilon_0,
\end{align}
then
\[
\left|
\det\bigl(
\partial_{\beta_1}f_\varepsilon(0),
\dots,
\partial_{\beta_n}f_\varepsilon(0)
\bigr)
\right|
\geq
\frac{c_{\mathrm{piv}}}{2}.
\]
\end{lem}
\begin{proof}
Since \(f\in C^N\) and \(\overline B\) is compact, for each
\(i=1,\dots,n\) we have
\[
|\partial_{\beta_i}f(0)|
\le \|f\|_{C^l(\overline B)}
< M.
\]
Let \(f_\varepsilon\) be as in the statement. Since \(l\le N\) and
\(0<\varepsilon\le \varepsilon_0<1\), we have
\[
\|f_\varepsilon\|_{C^l(\overline B)}
\le
\|f\|_{C^l(\overline B)}
+
\|f_\varepsilon-f\|_{C^l(\overline B)}
<
\|f\|_{C^l(\overline B)}+1
=
M.
\]
In particular,
\begin{align}\label{8}
|\partial_{\beta_i}f_\varepsilon(0)|
\le
M
\end{align}
for each \(i=1,\dots,n\). By the multi-linearity of the determinant,
\begin{align}\label{6}
&
\det(
\partial_{\beta_1}f_\varepsilon(0),
\dots,
\partial_{\beta_n}f_\varepsilon(0))
-
\det(
\partial_{\beta_1}f(0),
\dots,
\partial_{\beta_n}f(0))
\notag
\\
&=
\sum_{k=1}^{n}
\det(
\partial_{\beta_1}f_\varepsilon(0),
\dots,
\partial_{\beta_{k-1}}f_\varepsilon(0),
\partial_{\beta_k}f_\varepsilon(0)
-
\partial_{\beta_k}f(0),
\notag
\\
&\qquad\qquad\qquad\qquad\qquad\qquad\qquad\quad
\partial_{\beta_{k+1}}f(0),
\dots,
\partial_{\beta_n}f(0)
).
\end{align}
For column vectors \(C_1,\dots,C_n\in\mathbb R^n\), the Leibniz formula gives
\[
\det(C_1,\dots,C_n)
=
\sum_{\sigma\in S_n}
\operatorname{sgn}(\sigma)
\prod_{i=1}^{n}
(C_i)_{\sigma(i)},
\]
where \((C_i)_j\) denotes the \(j\)-th component of \(C_i\). Together with the bounds \eqref{7} and \eqref{8}, the absolute value of each determinant on the RHS of \eqref{6} is bounded by $n!\,\varepsilon M^{n-1}$. Hence
\[
\left|
\det(
\partial_{\beta_1}f_\varepsilon(0),
\dots,
\partial_{\beta_n}f_\varepsilon(0))
-
\det(
\partial_{\beta_1}f(0),
\dots,
\partial_{\beta_n}f(0))
\right|
\le
nn!\,\varepsilon M^{n-1}.
\]
Since
\[
0<\varepsilon\le\varepsilon_0<
\frac{c_{\mathrm{piv}}}
{2nn!M^{n-1}},
\]
we have
\[
\left|
\det(
\partial_{\beta_1}f_\varepsilon(0),
\dots,
\partial_{\beta_n}f_\varepsilon(0))
-
\det(
\partial_{\beta_1}f(0),
\dots,
\partial_{\beta_n}f(0))
\right|
\le
\frac{c_{\mathrm{piv}}}{2}.
\]
The reverse triangle inequality therefore yields
\begin{align*}
\left|
\det(
\partial_{\beta_1}f_\varepsilon(0),
\dots,
\partial_{\beta_n}f_\varepsilon(0))
\right|
&\ge
-
\frac{c_{\mathrm{piv}}}{2}
+
\left|
\det(
\partial_{\beta_1}f(0),
\dots,
\partial_{\beta_n}f(0))
\right|
\\
&\ge
c_{\mathrm{piv}}
-
\frac{c_{\mathrm{piv}}}{2}
=
\frac{c_{\mathrm{piv}}}{2}.
\qedhere
\end{align*}
\end{proof}
We now shift our attention from the determinant of the pivot derivatives on the manifold to the Wronskian determinant of the associated fibered curves. The pivot data constructed above will provide the key ingredient in identifying its leading non-vanishing term.
\subsubsection{Uniform Wronskian lower bound for polynomial term}
Although each perturbation $f_\varepsilon$ is analytic, we shall only consider its Taylor expansion up to order $N$. The reason is that the uniform information inherited from the original manifold $f$ is encoded in the estimate
\[
\|f_\varepsilon-f\|_{C^N(\overline B)}<\varepsilon,
\]
which provides control of derivatives only up to order $N$.
Since the original manifold is assumed merely to be $C^N$, no uniform control on higher-order derivatives is available. Consequently, in order to obtain estimates which are uniform in
$0<\varepsilon\le\varepsilon_0$, it is necessary to isolate the Taylor terms of degree at most $N$ and treat the higher-order terms as a remainder.

Recall that $f_\varepsilon=(f_{\varepsilon,1},\dots,f_{\varepsilon,n})$ and $f_{\varepsilon,0}=1$. For each $i=0,1,\dots,n$, Taylor's theorem gives
\[
f_{\varepsilon,i}(x)=\sum_{\vert \alpha\vert \leq N}\frac{\partial_\alpha f_{\varepsilon,i}(0)}{\alpha!}x^\alpha+R_{\varepsilon,i}(x),
\]
where $R_{\varepsilon,i}(x)$ denotes the Taylor remainder. Substituting the parametrisation (\ref{5}) and writing $x=\gamma_u(t)$, define
\begin{align}\label{14}
\phi_{\varepsilon,u,i}(t) &:=f_{\varepsilon,i}(\gamma_u(t))=\sum_{\vert \alpha\vert \leq N}\frac{\partial_\alpha f_{\varepsilon,i}(0)}{\alpha!}(\gamma_u(t))^\alpha+R_{\varepsilon,i}(\gamma_u(t))\notag\\
&=\sum_{\vert \alpha\vert \leq N}\Lambda_\alpha^{(i)}(\varepsilon)\cdot u_1^{\alpha_1}\cdots u_m^{\alpha_m}t^{\alpha_1a_1+\cdots+\alpha_ma_m}+R_{\varepsilon,i}(\gamma_u(t))\notag\\
&=\sum_{\vert \alpha\vert \leq N}\Lambda_\alpha^{(i)}(\varepsilon)\cdot u^\alpha\cdot t^{e_d(\alpha)}+R_{\varepsilon,i}(\gamma_u(t))\notag\\
&=P_{\varepsilon,u,i}(t)+E_{\varepsilon,u,i}(t),
\end{align}
where, by definition, $e_d(\alpha)=\alpha_1a_1+\cdots+\alpha_ma_m=\sum_{j=1}^{m}\alpha_j(d^{j-1}+d^m)$.

The goal of this subsection is to establish a uniform lower bound for the Wronskian determinant associated with the polynomial part $P_{\varepsilon,u,i}$. The quantitative non-degeneracy encoded by the pivot multi-indices $\beta_1,\dots,\beta_n$ will allow us to identify the leading non-vanishing term of the Wronskian expansion, yielding a uniform lower bound independent of $0<\varepsilon\le\varepsilon_0$.

The Wronskian determinant is closely related to linear independence in the analytic setting. Therefore, before establishing a uniform lower bound for the Wronskian of the polynomial part, we first show that the polynomial family $P_{\varepsilon,u,0},\dots,P_{\varepsilon,u,n}$ satisfies a quantitative form of linear independence.
\begin{lem}[Quantitative linear independence of the polynomial term]\label{lem5}
Let \(P_{\varepsilon,u,i}\) be defined by \eqref{14}, and let
\(U_0\subset(\mathbb R^\times)^m\) be compact as in Theorem \ref{upnf}. Then there exists a constant \(\mu_1>0\), depending only on \(U_0\), such that for every \(u\in U_0\), every \(0<\varepsilon\le\varepsilon_0\), and every \(c=(c_0,\dots,c_n)\in\mathbb R^{n+1}\) with \(|c|=1\), there exists \(\alpha\in\{0,\beta_1,\dots,\beta_n\}\) satisfying
\[
\bigl|
c\cdot(u^\alpha\Lambda_\alpha(\varepsilon))
\bigr|
\ge \mu_1.
\]
In particular, the polynomials $P_{\varepsilon,u,0},\dots,P_{\varepsilon,u,n}$ are linearly independent over \(\mathbb R\) for every \(u\in U_0\) and every \(0<\varepsilon\le\varepsilon_0\).
\end{lem}
\begin{proof}
By Lemma \ref{lem3}, we have
\begin{align*}
\left|
\det(
\Lambda_0(\varepsilon),
\Lambda_{\beta_1}(\varepsilon),
\dots,
\Lambda_{\beta_n}(\varepsilon))
\right|
&=
\frac{
\left|
\det(
\partial_{\beta_1}f_\varepsilon(0),
\dots,
\partial_{\beta_n}f_\varepsilon(0))
\right|
}{
\beta_1!\cdots\beta_n!
}\\
&\ge
\frac{
c_{\mathrm{piv}}
}{
2\beta_1!\cdots\beta_n!
}.
\end{align*}
Define
\begin{align}\label{18}
\Delta_0
:=
\frac{
c_{\mathrm{piv}}
}{
2\beta_1!\cdots\beta_n!
}>0.
\end{align}
Since
\(
\|f_\varepsilon-f\|_{C^l(\overline B)}
<
\varepsilon_0
<
1,
\)
we have
\[
\|f_\varepsilon\|_{C^l(\overline B)}
<
\|f\|_{C^l(\overline B)}+1.
\]
Therefore every entry of the matrix
\[
A_\varepsilon
=
\begin{pmatrix}
\Lambda_0(\varepsilon)\\
\Lambda_{\beta_1}(\varepsilon)\\
\vdots\\
\Lambda_{\beta_n}(\varepsilon)
\end{pmatrix}
\]
is uniformly bounded. Hence there exists a constant \(M_A>0\), independent of
\(0<\varepsilon\le\varepsilon_0\), such that
\[
\|A_\varepsilon\|_{\mathrm{op}}
\le
M_A.
\]
Denote
\[
\sigma_{\mathrm{min}}(A_\varepsilon)
=
\inf_{|v|=1}|A_\varepsilon v|.
\]
Then \(A_\varepsilon\) is invertible if and only if
\(\sigma_{\mathrm{min}}(A_\varepsilon)>0\). For a real \((n+1)\times(n+1)\) matrix \(A_\varepsilon\), its singular values are defined by
\[
\sigma_i(A_\varepsilon)
=
\sqrt{\lambda_i(A_\varepsilon^TA_\varepsilon)},
\]
where \(\lambda_i(A_\varepsilon^TA_\varepsilon)\) are the eigenvalues of the positive semi-definite matrix \(A_\varepsilon^TA_\varepsilon\). We order them as
\[
\sigma_1(A_\varepsilon)\ge\sigma_2(A_\varepsilon)\ge\cdots\ge\sigma_{n+1}(A_\varepsilon)\ge0,
\]
with $\sigma_1(A_\varepsilon)=\|A_\varepsilon\|_{\mathrm{op}}$ and $\sigma_{\mathrm{min}}(A_\varepsilon)=\sigma_{n+1}(A_\varepsilon)$. Moreover,
\[
|\det A_\varepsilon|
=
\prod_{i=1}^{n+1}\sigma_i(A_\varepsilon).
\]
Recall 
\[
\Vert A_\varepsilon\Vert_{\mathrm{op}}=\sup_{\vert v\vert=1}\vert A_\varepsilon v\vert\leq M_A,
\]
so $\sigma_1(A_\varepsilon),\dots,\sigma_n(A_\varepsilon),\sigma_{\text{min}}(A_\varepsilon)\leq M_A$. Hence
\[
\vert \det A_\varepsilon\vert\leq M_A^n\cdot \sigma_{\text{min}}(A_\varepsilon).
\]
Therefore
\[
\sigma_{\text{min}}(A_\varepsilon)\geq \frac{\vert \det A_\varepsilon\vert}{M^n_A}\geq \frac{\Delta_0}{M_A^n}.
\]
Hence, for every $c\in \mathbb{R}^{n+1}$ with $\vert c\vert=1$, we have
\[
\vert A_\varepsilon c\vert=\left\vert\begin{pmatrix}
\Lambda_0(\varepsilon)\cdot c\\
\vdots\\
\Lambda_{\beta_n}(\varepsilon)\cdot c
\end{pmatrix}\right\vert=((\Lambda_0(\varepsilon)\cdot c)^2+\cdots +(\Lambda_{\beta_n}(\varepsilon)\cdot c)^2)^{1/2}\geq \frac{\Delta_0}{M_A^n}>0.
\]
The above estimate implies that
\begin{align}\label{9}
\max_{\alpha\in \lbrace 0,\beta_1,\dots,\beta_n\rbrace}\vert \Lambda_\alpha(\varepsilon)\cdot c\vert \geq \frac{\Delta_0}{\sqrt{n+1}M_A^n}=:\mu_0>0.
\end{align}

Recall from \eqref{14} that
\[
P_{\varepsilon,u}(t)
=
(P_{\varepsilon,u,0}(t),\dots,P_{\varepsilon,u,n}(t))
=
\sum_{|\alpha|\le N}
u^\alpha
\Lambda_\alpha(\varepsilon)
t^{e_d(\alpha)}.
\]
The coefficient vectors of the polynomial family are therefore
\(u^\alpha\Lambda_\alpha(\varepsilon)\). Since \(U_0\subset(\mathbb R^\times)^m\) is compact,
\begin{align}\label{17}
\rho
:=
\min
\left\{
1,
\inf_{u\in U_0}\min_{1\le j\le m}|u_j|
\right\}
>0.
\end{align}
Hence, for any \(u\in U_0\) and any multi-index \(\alpha\) with \(|\alpha|\le l\),
\[
|u^\alpha|
=
\prod_{j=1}^m|u_j|^{\alpha_j}
\ge
\rho^{|\alpha|}
\ge
\rho^l.
\]
Combining this with \eqref{9} gives
\begin{align*}
\max_{\alpha\in\{0,\beta_1,\dots,\beta_n\}}
|c\cdot(u^\alpha\Lambda_\alpha(\varepsilon))|
&\ge
\rho^l
\max_{\alpha\in\{0,\beta_1,\dots,\beta_n\}}
|c\cdot\Lambda_\alpha(\varepsilon)|\\
&\ge
\rho^l\mu_0
=:
\mu_1>0.
\end{align*}

It remains to prove linear independence. Suppose
\[
\sum_{i=0}^n
c_iP_{\varepsilon,u,i}(t)
\equiv0
\]
for some
\(c=(c_0,\dots,c_n)\neq0\).
Then
\[
\sum_{|\alpha|\le N}
\bigl(
c\cdot(u^\alpha\Lambda_\alpha(\varepsilon))
\bigr)
t^{e_d(\alpha)}
\equiv0.
\]
Normalising \(c\) so that \(|c|=1\), the estimate above yields
\[
|c\cdot(u^\alpha\Lambda_\alpha(\varepsilon))|
\ge
\mu_1
\]
for some
\(\alpha\in\{0,\beta_1,\dots,\beta_n\}\).
Since \(|\alpha|\le l\), Lemma \ref{lem4} implies that
\(e_d(\alpha)\) is distinct from every other exponent arising in the above expansion. Hence the corresponding monomial cannot cancel with any other term, a contradiction. Therefore $P_{\varepsilon,u,0},
\dots,
P_{\varepsilon,u,n}$ are linearly independent over \(\mathbb R\).
\end{proof}
Now define the extended \((n+1)\times(n+1)\) Wronskian of the polynomial term by
\begin{align}\label{57}
W_{\varepsilon,u}^P(t)
:=
W(P_{\varepsilon,u,0},P_{\varepsilon,u,1},\dots,P_{\varepsilon,u,n})(t)
=
\det
\begin{pmatrix}
P_{\varepsilon,u,0}(t) & \cdots & P_{\varepsilon,u,n}(t)\\
P'_{\varepsilon,u,0}(t) & \cdots & P'_{\varepsilon,u,n}(t)\\
\vdots & \cdots & \vdots\\
P^{(n)}_{\varepsilon,u,0}(t) & \cdots & P^{(n)}_{\varepsilon,u,n}(t)
\end{pmatrix}.
\end{align}
Since \(f_{\varepsilon,0}\equiv 1\), we have \(P_{\varepsilon,u,0}\equiv 1\). Hence the extended Wronskian reduces to the \(n\times n\) determinant
\begin{align}\label{16}
W_{\varepsilon,u}^P(t)
=
\det
\bigl(
P_{\varepsilon,u,j}^{(i)}(t)
\bigr)_{1\le i,j\le n}.
\end{align}
The functions \(P_{\varepsilon,u,0},\dots,P_{\varepsilon,u,n}\) are polynomials, and hence analytic. By Lemma \ref{lem5}, they are linearly independent over \(\mathbb R\). Since, for analytic functions, linear independence is equivalent to the non-vanishing of the Wronskian, we have $W_{\varepsilon,u}^P\not\equiv 0$.

However, quantitative linear independence alone is not enough for our purposes. A compactness argument would yield a uniform lower bound for the Wronskian on any fixed interval $J=[0,t_*]$, but the resulting constant depends on $t_*$. Since $t_*$ must later be shrunk in order to control the remainder term, we instead seek an estimate in which the constant is independent of $t_*$. Rather than working directly with the size of the Wronskian on $[0,t_*]$, we analyse the polynomial Wronskian explicitly. Using the pivot construction, we identify a distinguished non-vanishing leading term whose coefficient admits a uniform positive lower bound. This coefficient level control is independent of $t_*$ and yields an explicit lower bound after rescaling.
\begin{prop}[Uniform Wronskian lower bound for the polynomial term]\label{prop2}
Let \(P_{\varepsilon,u,i}\) be defined by \eqref{14}, and let
\(U_0\subset(\mathbb R^\times)^m\) be compact. Then there exist constants
\(0< t_0\leq 1\) and \(c_P>0\), independent of \(0<\varepsilon\le\varepsilon_0\) and \(u\in U_0\), such that for every $0<t_*\leq t_0$, every \(0<\varepsilon\le\varepsilon_0\)
and every \(u\in U_0\), there exists \(t_{\varepsilon,u}\in[0,t_*]\)
satisfying
\[
\bigl|W_{\varepsilon,u}^P(t_{\varepsilon,u})\bigr|
\ge c_P t_*^{Q_{\mathrm{good}}},
\]
where
\begin{align}\label{63}
Q_{\mathrm{good}}
:=
P_n-\frac{n(n+1)}2,
\qquad
P_n:=p_1+\cdots+p_n.
\end{align}
Here \(p_1,\dots,p_n\) denote the pivot exponents constructed in Lemma~\ref{lemp}. Equivalently,
\begin{align}\label{94}
\sup_{t\in[0,t_*]}
\bigl|W_{\varepsilon,u}^P(t)\bigr|
\ge c_P t_*^{Q_{\mathrm{good}}}
\end{align}
uniformly for \(0<\varepsilon\le\varepsilon_0\) and \(u\in U_0\).
\end{prop}
\begin{proof}
\textbf{Step 1. Construction of the leading term.}

We first identify the leading non-vanishing term of the polynomial Wronskian. The pivot construction of Lemma~\ref{lemp} provides the key ingredient in identifying a distinguished leading term with non-zero coefficient. We then use a compactness argument to show that this coefficient is uniformly bounded away from zero. Recall from \eqref{14} that
\[
P_{\varepsilon,u}(t)
=
(P_{\varepsilon,u,0}(t),\dots,P_{\varepsilon,u,n}(t))
=
\sum_{|\alpha|\le N}
u^\alpha\Lambda_\alpha(\varepsilon)t^{e_d(\alpha)}.
\]
Since \(d\geq d_0\ge N\), Lemma~\ref{lem4} implies that \(e_d\) is injective on $\{\alpha\in\mathbb Z_{\ge0}^m:|\alpha|\le N\}$. Hence we may write
\begin{align}\label{56}
P_{\varepsilon,u}(t)
=
\sum_q
v_q(\varepsilon,u)t^q,
\end{align}
where the sum is finite, \(q\) ranges over the exponents
\(e_d(\alpha)\) with \(|\alpha|\le N\), and \(v_q(\varepsilon,u)\in\mathbb R^{n+1}\). By multi-linearity of the determinant, we have
\begin{align}\label{55}
&
\left|
\det\bigl(
u^0\Lambda_0(\varepsilon),
u^{\beta_1}\Lambda_{\beta_1}(\varepsilon),
\dots,
u^{\beta_n}\Lambda_{\beta_n}(\varepsilon)
\bigr)
\right|\notag
\\
&=
|u^{\beta_1+\cdots+\beta_n}|
\left|
\det\bigl(
\Lambda_0(\varepsilon),
\Lambda_{\beta_1}(\varepsilon),
\dots,
\Lambda_{\beta_n}(\varepsilon)
\bigr)
\right|\notag
\\
&\ge
\rho^{|\beta_1|+\cdots+|\beta_n|}
\Delta_0
>0,
\end{align}
where \(\rho\) is defined by \eqref{17}, \(\Delta_0\) is defined by \eqref{18}. Since \(u\in U_0\subset(\mathbb R^\times)^m\), each \(u^{\beta_i}\) is non-zero. Thus the vectors $u^0\Lambda_0(\varepsilon),
u^{\beta_1}\Lambda_{\beta_1}(\varepsilon),
\dots,
u^{\beta_n}\Lambda_{\beta_n}(\varepsilon)$ span \(\mathbb R^{n+1}\) if and only if $\Lambda_0(\varepsilon),
\Lambda_{\beta_1}(\varepsilon),
\dots,
\Lambda_{\beta_n}(\varepsilon)$ do. Applying the greedy construction used in Lemma~\ref{lemp} to $\{\Lambda_\alpha(\varepsilon):|\alpha|\le l\}$, we obtain multi-indices
\[
\tilde\beta_0(\varepsilon)=0,\tilde\beta_1(\varepsilon),\dots,\tilde\beta_n(\varepsilon),
\]
with associated exponents
\begin{align}\label{59}
0=\tilde p_0(\varepsilon)<\tilde p_1(\varepsilon)<\cdots<\tilde p_n(\varepsilon),\quad
\tilde p_i(\varepsilon)=e_d(\tilde\beta_i(\varepsilon)),
\end{align}
which are minimal among all spanning collections. By Lemma~\ref{lemp} and the lower bound \eqref{18}, the vectors $\Lambda_0(\varepsilon),
\Lambda_{\beta_1}(\varepsilon),
\dots,
\Lambda_{\beta_n}(\varepsilon)$ form a spanning collection. Hence, by the minimality of the perturbed pivots,
\begin{align}\label{58}
\tilde p_i(\varepsilon)\le p_i,
\quad
0\le i\le n.
\end{align}
Although the polynomial expansion ranges over all multi-indices with \(|\alpha|\le N\), whereas the pivot construction is carried out only on \(\{|\alpha|\le l\}\), this causes no difficulty. Indeed, 
\begin{align}\label{61}
|\alpha|<|\alpha'|
\quad\Longrightarrow\quad
e_d(\alpha)<e_d(\alpha')
\end{align}
for all \(d\ge d_0\) and all \(|\alpha|,|\alpha'|\le N\), since if $\vert \alpha\vert<\vert \alpha'\vert\leq N\leq d$, then $e_d(\alpha)\leq \vert \alpha\vert d^m+\vert \alpha\vert d^{m-1}\leq (\vert \alpha\vert+1)d^m<e_d(\alpha')$. Consequently, every exponent arising from \(|\alpha|>l\) is larger than every exponent arising from \(|\alpha|\le l\). Hence the perturbed pivot exponents $\tilde p_0(\varepsilon),\dots,\tilde p_n(\varepsilon)$ remain minimal among all spanning collections occurring in the polynomial expansion.

Using the extended Wronskian representation in \eqref{57}, together with the polynomial expansion in \eqref{56}, we obtain
\[
W_{\varepsilon,u}^P(t)=\det\left( \sum_q v_q(\varepsilon,u)t^q,\sum_q v_q(\varepsilon,u)(t^q)',\dots,\sum_q v_q(\varepsilon,u)(t^q)^{(n)}   \right).
\]
Expanding this determinant by multi-linearity in the columns gives
\[
W_{\varepsilon,u}^P(t)=\sum_{q_0,\dots,q_n}\det\left( v_{q_0}(\varepsilon,u),\dots,v_{q_n}(\varepsilon,u) \right) t^{q_0} (t^{q_1})'\cdots (t^{q_n})^{(n)}.
\]
If two exponents coincide, say $q_i=q_j$, then $v_{q_i}(\varepsilon,u)=v_{q_j}(\varepsilon,u)$, so the determinant has two identical columns and therefore vanishes. Hence only tuples with pairwise distinct exponents contribute. The above then becomes
\begin{align}\label{19}
&W_{\varepsilon,u}^P(t)
=\sum_{\substack{q_0<\cdots <q_n\\ \sigma\in S_{n+1}}}\det\left( v_{q_{\sigma(0)}}(\varepsilon,u),\dots,v_{q_{\sigma(n)}}(\varepsilon,u)\right)\cdot t^{q_{\sigma(0)}}\cdots (t^{q_{\sigma(n)}})^{(n)}\notag\\
&=\sum_{q_0<\cdots <q_n}\det\left( v_{q_0}(\varepsilon,u),\dots,v_{q_n}(\varepsilon,u)\right)\cdot\left( \sum_{\sigma\in S_{n+1}}\text{sgn}(\sigma) t^{q_{\sigma(0)}}\cdots (t^{q_{\sigma(n)}})^{(n)} \right)\notag\\
&=\sum_{0=q_0<\cdots<q_n}
\det\bigl(
v_{q_0}(\varepsilon,u),
\dots,
v_{q_n}(\varepsilon,u)
\bigr)
\cdot
\det
\begin{pmatrix}
t^{q_0} & \cdots & t^{q_n}\\
(t^{q_0})' & \cdots & (t^{q_n})'\\
\vdots & \ddots & \vdots\\
(t^{q_0})^{(n)} & \cdots & (t^{q_n})^{(n)}
\end{pmatrix},
\end{align}
where \(0=q_0<\cdots<q_n\) range over the exponents
\(e_d(\alpha)\) with \(|\alpha|\le N\). The restriction to $q_0=0$ is justified by the fact that, since $P_{\varepsilon,u,0}\equiv1$, the coefficient vectors satisfy $v_0(\varepsilon,u)=(1,*,\dots,*)^T$ and $v_q(\varepsilon,u)=(0,*,\dots,*)^T$ for $q>0$. Hence only terms containing the exponent $0$ can have non-zero coefficient determinant. 

Since \(\tilde\beta_i(\varepsilon)\neq0\) for every \(1\le i\le n\), the definition of
\(e_d\) in \eqref{12} gives
\[
\tilde p_i(\varepsilon)
=
e_d(\tilde\beta_i(\varepsilon))
\ge d^m.
\]
Hence, since $d\geq d_0\geq N\geq n$ and therefore \(d^m\ge n\), we have
\begin{align*}
\tilde p_i(\varepsilon)\ge n,
\quad
1\le i\le n.
\end{align*}
In particular, the derivatives
\((t^{\tilde p_i(\varepsilon)})^{(k)}\) are well defined for every
\(1\le k\le n\). 

We now focus on the contribution to \eqref{19} corresponding to the perturbed pivot exponents in \eqref{59}. For this term, since \(\tilde p_0(\varepsilon)=0\), the first column of the second determinant in \eqref{19} is $(1,0,\dots,0)^T$. Hence this determinant reduces to an \(n\times n\) Wronskian. Therefore the corresponding contribution is
\begin{align}\label{20}
\det\bigl(
v_{\tilde p_0(\varepsilon)}(\varepsilon,u),
\dots,
v_{\tilde p_n(\varepsilon)}(\varepsilon,u)
\bigr)
\cdot
W(
(t^{\tilde p_1(\varepsilon)})',
\dots,
(t^{\tilde p_n(\varepsilon)})'
).
\end{align}
Second term of \eqref{20} then becomes
\[
W((t^{\tilde p_1(\varepsilon)})',\dots,(t^{\tilde p_n(\varepsilon)})')
=\det\begin{pmatrix}
(\tilde p_1(\varepsilon))_1 t^{\tilde p_1(\varepsilon)-1} & \cdots & (\tilde p_n(\varepsilon))_1 t^{\tilde p_n(\varepsilon)-1}\\
\vdots & \cdots & \vdots\\
(\tilde p_1(\varepsilon))_nt^{\tilde p_1(\varepsilon)-n} & \cdots & (\tilde p_n(\varepsilon))_n t^{\tilde p_n(\varepsilon)-n}
\end{pmatrix},
\]
where
\[
(\tilde p_i(\varepsilon))_k
=
\tilde p_i(\varepsilon)(\tilde p_i(\varepsilon)-1)\cdots(\tilde p_i(\varepsilon)-k+1).
\]
Again, by multi-linearity of the determinant, factor out the exponent terms \(t^{\tilde p_i(\varepsilon)}\) from each column and then the exponent terms \(t^{-k}\) from each row. The above becomes
\[
t^{\tilde p_1(\varepsilon)+\cdots+\tilde p_n(\varepsilon)}
\cdot
t^{-1-\cdots-n}
\cdot
\det\begin{pmatrix}
(\tilde p_1(\varepsilon))_1 & \cdots & (\tilde p_n(\varepsilon))_1\\
\vdots & \cdots & \vdots\\
(\tilde p_1(\varepsilon))_n & \cdots & (\tilde p_n(\varepsilon))_n
\end{pmatrix}.
\]
Due to the definition of \((\tilde p_i(\varepsilon))_k\), we may factor out \(\tilde p_i(\varepsilon)\) from the \(i\)-th column of the above matrix. Again, by multi-linearity of the determinant, the above becomes
\[
t^{\tilde p_1(\varepsilon)+\cdots+\tilde p_n(\varepsilon)-\frac{n(n+1)}2}
\cdot
\prod_{i=1}^{n}\tilde p_i(\varepsilon)
\cdot
\det\bigl((\tilde p_i(\varepsilon)-1)_{j-1}\bigr)_{1\le i,j\le n}.
\]
It remains to evaluate the determinant. Since
\[
(\tilde p_i(\varepsilon)-1)_0=1,
\]
we have
\begin{align*}
&\det\begin{pmatrix}
1 & \cdots & 1\\
\tilde p_1(\varepsilon)-1 & \cdots & \tilde p_n(\varepsilon)-1\\
(\tilde p_1(\varepsilon)-1)(\tilde p_1(\varepsilon)-2) & \cdots &
(\tilde p_n(\varepsilon)-1)(\tilde p_n(\varepsilon)-2)\\
\vdots & \cdots & \vdots\\
(\tilde p_1(\varepsilon)-1)\cdots(\tilde p_1(\varepsilon)-n+1)
& \cdots &
(\tilde p_n(\varepsilon)-1)\cdots(\tilde p_n(\varepsilon)-n+1)
\end{pmatrix}\\
&=
\det\begin{pmatrix}
1 & \cdots & 1\\
\tilde p_1(\varepsilon)-1 & \cdots & \tilde p_n(\varepsilon)-1\\
\tilde p_1(\varepsilon)^2-3\tilde p_1(\varepsilon)+2 &
\cdots &
\tilde p_n(\varepsilon)^2-3\tilde p_n(\varepsilon)+2\\
\vdots & \cdots & \vdots\\
\tilde p_1(\varepsilon)^{n-1}+O(\tilde p_1(\varepsilon)^{n-2})
&
\cdots
&
\tilde p_n(\varepsilon)^{n-1}+O(\tilde p_n(\varepsilon)^{n-2})
\end{pmatrix}\\
&=
\det\begin{pmatrix}
1 & \tilde p_1(\varepsilon) & \cdots &
\tilde p_1(\varepsilon)^{n-1}\\
1 & \tilde p_2(\varepsilon) & \cdots &
\tilde p_2(\varepsilon)^{n-1}\\
\vdots & \vdots & \cdots & \vdots\\
1 & \tilde p_n(\varepsilon) & \cdots &
\tilde p_n(\varepsilon)^{n-1}
\end{pmatrix}
\cdot
\det\begin{pmatrix}
1 & -1 & 2 & -6 & \cdots\\
0 & 1 & -3 & 11 & \cdots\\
0 & 0 & 1 & -6 & \cdots\\
0 & 0 & 0 & 1 & \cdots\\
\vdots & \vdots & \vdots & \vdots & \ddots
\end{pmatrix}\\
&=
\prod_{1\le i<k\le n}
\bigl(\tilde p_k(\varepsilon)-\tilde p_i(\varepsilon)\bigr)
\neq0.
\end{align*}
This is non-zero by the injectivity of \(e_d\). Combining the above computations, \eqref{20} becomes
\begin{align}\label{22}
&
\det\bigl(
v_{\tilde p_0(\varepsilon)}(\varepsilon,u),
\dots,
v_{\tilde p_n(\varepsilon)}(\varepsilon,u)
\bigr)
\cdot
\prod_{i=1}^{n}\tilde p_i(\varepsilon)
\cdot
\prod_{1\le i<k\le n}
\bigl(\tilde p_k(\varepsilon)-\tilde p_i(\varepsilon)\bigr)\notag
\\
&\hspace{6cm}\cdot\,
t^{\tilde p_1(\varepsilon)+\cdots+\tilde p_n(\varepsilon)-\frac{n(n+1)}2}.
\end{align}
Since
\[
\det\bigl(
v_{\tilde p_0(\varepsilon)}(\varepsilon,u),
\dots,
v_{\tilde p_n(\varepsilon)}(\varepsilon,u)
\bigr)\neq0,
\]
the coefficient of the monomial in \eqref{22} is non-zero. By \eqref{58}, the exponent of \(t\) in \eqref{22} satisfies
\[
\tilde{Q}_{\mathrm{good}}(\varepsilon):=\tilde p_1(\varepsilon)+\cdots+\tilde p_n(\varepsilon)-\frac{n(n+1)}2
\le
Q_{\mathrm{good}}.
\]
Knowing that the contribution \eqref{22} is non-zero is not enough. We must also show that the exponent \(\tilde Q_{\mathrm{good}}(\varepsilon)\) is unique in the expansion \eqref{19}, so that no cancellation can occur with other summands.

By the minimality of the perturbed pivot exponents $0=\tilde p_0(\varepsilon)<\cdots<\tilde p_n(\varepsilon)$, every tuple of exponents $\tilde p_0'(\varepsilon),\dots,\tilde p_n'(\varepsilon)$ with non-zero determinant in \eqref{19} satisfies, after re-ordering,
\[
\tilde p_i'(\varepsilon)\ge\tilde p_i(\varepsilon),
\qquad
0\le i\le n.
\]
Hence, if
\[
\tilde p_0'(\varepsilon)+\cdots+\tilde p_n'(\varepsilon)
=
\tilde p_0(\varepsilon)+\cdots+\tilde p_n(\varepsilon),
\]
then necessarily
\[
\tilde p_i'(\varepsilon)=\tilde p_i(\varepsilon),
\qquad
0\le i\le n.
\]
Since each difference $\tilde p_i'(\varepsilon)-\tilde{p_i}(\varepsilon)$ is non-negative, equality of the sums forces equality coordinate wise. Together with the injectivity of \(e_d\), this implies that no other summand in \eqref{19} contributes to the monomial of degree \(\tilde Q_{\mathrm{good}}(\varepsilon)\). Combining this with \eqref{22}, we conclude that the monomial of degree \(\tilde Q_{\mathrm{good}}(\varepsilon)\) is unique in the expansion \eqref{19} and has non-zero coefficient.

\textbf{Step 2. Uniformisation via compactness.}

We now establish a uniform lower bound for the preceding non-vanishing coefficient. Since \(S^n\) is compact and the maximum of finitely many continuous functions is continuous, the infimum is attained. Hence \(\inf\) may be replaced by \(\min\). By Lemma \ref{lem5}, we have
\begin{align}\label{24}
\min_{\vert c\vert =1}\max_{\alpha\in \lbrace 0,\beta_1,\dots,\beta_n\rbrace}\bigl|
c\cdot(u^\alpha\Lambda_\alpha(\varepsilon))
\bigr|
\ge \mu_1>0.
\end{align}
Arguing as in the proof of Lemma~\ref{lem3}, choose $M_P>0$ such that, for every
\(0<\varepsilon\le \varepsilon_0\), every \(u\in U_0\), and every
\(|\alpha|\le N\), the coefficient vector $u^\alpha\Lambda_\alpha(\varepsilon)$ satisfies
\begin{align}\label{25}
|u^\alpha\Lambda_\alpha(\varepsilon)|\le M_P.
\end{align}
Such an \(M_P\) exists because \(U_0\) is compact and
\(\|f_\varepsilon-f\|_{C^N(\overline B)}<\varepsilon_0\). Indeed,
\[
\vert \partial_\alpha f_{\varepsilon,i}(0)\vert\leq \Vert f_\varepsilon\Vert_{C^N(\overline{B})}\leq \Vert f\Vert_{C^N(\overline{B})}+\varepsilon_0.
\]
Let
\(\mathcal A\) denote the finite dimensional space of coefficient data consisting of $A=(a_\alpha)_{|\alpha|\le N}$, where $a_\alpha\in \mathbb R^{n+1}$. Define
\begin{align*}
K
=
\lbrace
A=(a_\alpha)_{|\alpha|\le N}\in\mathcal A:
|a_\alpha|\le M_P\ &\text{for all }|\alpha|\le N,\\
&\min_{|c|=1}
\max_{\alpha\in\{0,\beta_1,\dots,\beta_n\}}
|c\cdot a_\alpha|
\ge \mu_1
\rbrace.
\end{align*}
Then \(K\) is compact. Indeed, it's clearly bounded by construction. To see that $K$ is closed, let
\[
G(A)=\min_{|c|=1}
\max_{\alpha\in\{0,\beta_1,\dots,\beta_n\}}
|c\cdot a_\alpha|\quad\text{and}\quad H(A,c)=\max_{\alpha\in\{0,\beta_1,\dots,\beta_n\}}
|c\cdot a_\alpha|.
\]
We already know that \(H(A,c)\) is continuous in \((A,c)\). Since \(c\) ranges over the compact set \(S^n\), it follows that $G(A)=\min_{|c|=1}H(A,c)$ is continuous in \(A\). Then
\[
\lbrace A:\, G(A)\geq \mu_1\rbrace=G^{-1}([\mu_1,\infty))
\]
is closed because $[\mu_1,\infty)$ is closed. Therefore
\[
K
=
(\bigcap_{|\alpha|\le N}\lbrace|a_\alpha|\le M_P\rbrace)
\cap
\{A:G(A)\ge \mu_1\}
\]
is the intersection of finitely many closed sets, and therefore closed.

For every \(0<\varepsilon\le\varepsilon_0\) and every \(u\in U_0\), the actual coefficient data
\[
A_{\varepsilon,u}
=
\bigl(u^\alpha\Lambda_\alpha(\varepsilon)\bigr)_{|\alpha|\le N}
\]
belongs to \(K\), by (\ref{24}) and (\ref{25}). For \(A\in K\), define the polynomial family
\[
P_{A,i}(t)
=
\sum_{|\alpha|\le N}
a_\alpha^{(i)}t^{e_d(\alpha)},
\quad
0\le i\le n,
\]
where $a_\alpha=(a_\alpha^{(0)},\dots,a_\alpha^{(n)})$. Let
\[
W_A^P(t)
=
W(P_{A,0},\dots,P_{A,n})(t).
\]
Since each \(P_{A,i}\) has degree bounded independently of \(A\), there is a fixed integer \(Q_{\max}\) such that
\begin{align}\label{93}
W_A^P(t)
=
\sum_{q=0}^{Q_{\max}} b_q(A)t^q.
\end{align}
Each coefficient \(b_q(A)\) is a polynomial expression in the finitely many vectors \(a_\alpha\), and hence depends continuously on \(A\). Define
\[
\text{Coeff}_{\leq Q_{\text{good}}}(W_A^P)
=
(b_0(A),b_1(A),\dots,b_{Q_{\mathrm{good}}}(A))
\in \mathbb R^{Q_{\mathrm{good}}+1}.
\]

For every \(A\in K\), the defining condition of \(K\), arguing as in the proof of Lemma~\ref{lem5}, yields the quantitative linear independence of $a_0,a_{\beta_1},\dots,a_{\beta_n}$. In particular,
\[
\det(a_0,a_{\beta_1},\dots,a_{\beta_n})\neq0.
\]
Indeed, assume for contradiction that $a_0,a_{\beta_1},\dots,a_{\beta_n}$ do not span $\mathbb{R}^{n+1}$. Then their span is a proper subspace $V\subsetneq \mathbb{R}^{n+1}$. Therefore there exists a non-zero vector $c\in V^\perp$ with $\vert c\vert=1$ such that 
\[
c\cdot a_\alpha=0,\quad\forall\,\alpha\in\lbrace 0,\beta_1,\dots,\beta_n\rbrace.
\]
This contradicts the fact that
\[
\min_{|c|=1}
\max_{\alpha\in\{0,\beta_1,\dots,\beta_n\}}
|c\cdot a_\alpha|
\ge \mu_1.
\]

Applying the same greedy construction with \(a_\alpha\) in place of
\(u^\alpha\Lambda_\alpha(\varepsilon)\), let
\[
0=\tilde p_0(A)<\tilde p_1(A)<\cdots<\tilde p_n(A)
\]
be the resulting pivot exponents. Since
\(a_0,a_{\beta_1},\dots,a_{\beta_n}\) are linearly independent and
\(e_d(0)=0\), we have \(\tilde p_0(A)=0\), and by minimality,
\[
\tilde p_i(A)\le p_i,
\quad 0\le i\le n.
\]
By the same calculation as that leading to \eqref{22}, the pivot tuple
contributes to \(W_A^P\) a non-zero monomial of degree
\[
\tilde p_1(A)+\cdots+\tilde p_n(A)
-\frac{n(n+1)}2
\le Q_{\mathrm{good}}.
\]
By the same greedy minimality argument as in Step~1, this contribution
is unique in that degree. Hence
\[
\operatorname{Coeff}_{\le Q_{\mathrm{good}}}(W_A^P)\neq0
\quad\text{for every }A\in K.
\]

Therefore the continuous function
\[
A\mapsto |\text{Coeff}_{\leq Q_{\text{good}}}(W_A^P)|=\Big(\sum_{q=0}^{Q_{\text{good}}}\vert b_q(A)\vert^2\Big)^{1/2}
\]
is strictly positive on the compact set \(K\). Hence it attains a positive
minimum:
\[
B_1
:=
\min_{A\in K}|\text{Coeff}_{\leq Q_{\text{good}}}(W_A^P)|
>0.
\]
Consequently, for every \(0<\varepsilon\le\varepsilon_0\) and every
\(u\in U_0\),
\[
\left|
\bigl(
b_0(A_{\varepsilon,u}),
\dots,
b_{Q_{\mathrm{good}}}(A_{\varepsilon,u})
\bigr)
\right|
\ge B_1.
\]
Equivalently,
\begin{align}\label{27}
\left|
\operatorname{Coeff}_{\le Q_{\mathrm{good}}}
(W_{\varepsilon,u}^P)
\right|
\ge B_1.
\end{align}
The constant \(B_1\) is
independent of \(0<\varepsilon\le\varepsilon_0\), \(u\in U_0\), and \(t_*\). 

\textbf{Step 3. Passage to a Wronskian lower bound.}

Now let $J=[0,t_*]$. Since $J$ will later be shrunk in order to control the remainder term, we require the constants in the lower bound for $W_{\varepsilon,u}^P$ to be independent of $t_*$. We rescale $t$ by writing $t=t_*s$, $s\in [0,1]$, so
\[
p(t_*s):=\sum_{q\leq Q_{\text{good}}}b_q(\varepsilon,u)t_*^qs^q=t_*^{Q_{\text{good}}}\sum_{q\leq Q_{\text{good}}}b_q(\varepsilon,u)t_*^{q-Q_{\text{good}}}s^q.
\]
The coefficient vector of the rescaled polynomial is
\[
(b_0(\varepsilon,u)t_*^{-Q_{\text{good}}},b_1(\varepsilon,u)t_*^{1-Q_{\text{good}}},\dots,b_{Q_{\text{good}}}(\varepsilon,u)).
\]
Since $0<t_*\leq 1$ and $t_*^{q-Q_{\text{good}}}\geq 1$, its Euclidean norm is at least
\[
(b_0^2(\varepsilon,u)+\cdots+b_{Q_{\text{good}}}^2(\varepsilon,u))^{1/2},
\]
which is bounded below by $B_1$ in view of (\ref{27}). Hence, by the equivalence of the coefficient norm and the supremum norm on the finite dimensional space of polynomials of degree at most $Q_{\mathrm{good}}$, there exists a constant $C_0>0$ independent of $t_*$, $\varepsilon$ and $u$ such that
\begin{align}\label{26}
\sup_{t\in J}\Big\vert \sum_{q\leq Q_{\text{good}}}b_q(\varepsilon,u)t^q\Big\vert&=t_*^{Q_{\text{good}}}\cdot\sup_{s\in [0,1]}\Big\vert \sum_{q\leq Q_{\text{good}}}b_q(\varepsilon,u)t_*^{q-Q_{\text{good}}}s^q\Big\vert\notag\\
&\geq C_0 B_1 t_*^{Q_{\text{good}}}.
\end{align}

Without the quantitative control provided by $Q_{\mathrm{good}}$, it would still be possible to separate the dependence on $t_*$ from the constants, but only at the cost of factoring out the maximal exponent $Q_{\mathrm{max}}$. As we shall see later, the smoothness requirement on the manifold is determined precisely by the exponent that can be factored out. Consequently, to minimize the smoothness assumption, one must minimize this exponent. With the information available in the present framework, $Q_{\mathrm{good}}$ is essentially the smallest exponent that can be achieved.

For the remaining terms of $W_{\varepsilon,u}^P$, there exists $C_1>0$, independent of $t_*$, $\varepsilon$ and $u$, such that
\[
\Big\vert \sum_{Q_{\text{good}}<q\leq Q_{\text{max}}}b_q(\varepsilon,u)t^q\Big\vert \leq C_1t_*^{Q_{\text{good}}+1}
\]
for $t\in[0,t_*]$, because $t^q\le t_*^{Q_{\mathrm{good}}+1}$ for $q>Q_{\mathrm{good}}$, $t_*\le1$, and the coefficients $b_q(\varepsilon,u)$ are uniformly bounded. Set $t_0:=\min\lbrace 1,C_0B_1/(2C_1)\rbrace$. Since $C_0$, $B_1$ and $C_1$ are independent of $t_*$, for every $0<t_*\leq t_0$, we have
\[
C_1t_*^{Q_{\text{good}}+1}\leq \frac{1}{2}C_0B_1t_*^{Q_{\text{good}}}.
\]
Hence, combining (\ref{26}) with the above estimate and applying the triangle inequality, we obtain
\[
\sup_{t\in J}\vert W_{\varepsilon,u}^P(t)\vert\geq \frac{1}{2}C_0B_1t_*^{Q_{\text{good}}}.
\]
Therefore, for every $0<\varepsilon\leq \varepsilon_0$ and $u\in U_0$, there exists $t_{\varepsilon,u}\in J$ such that
\begin{align*}
\vert W_{\varepsilon,u}^P(t_{\varepsilon,u})\vert \geq \frac{1}{2}C_0B_1t_*^{Q_{\text{good}}},
\end{align*}
where the supremum is attained since $W_{\varepsilon,u}^P$ is continuous on the compact interval $J$. Denote
\[
c_P:=\frac{1}{2}C_0B_1.
\]
This completes the proof.
\end{proof}
\begin{rem}
The greedy construction is needed because the contribution associated with
\((\beta_1,\dots,\beta_n)\) need not be isolated at degree
\(Q_{\mathrm{good}}\). Indeed, from
\[
\sum_{i=1}^n e_d(\beta_i)=\sum_{i=1}^n e_d(\alpha_i),
\]
linearity and injectivity of \(e_d\) only imply
\[
\sum_{i=1}^n\beta_i=\sum_{i=1}^n\alpha_i,
\]
not equality of the multi-sets \(\{\beta_i\}\) and \(\{\alpha_i\}\). For
example,
\[
(2,0)+(0,2)=(1,1)+(1,1).
\]
Thus injectivity alone does not rule out cancellation at degree
\(Q_{\mathrm{good}}\) uniformly over all perturbed manifolds. The greedy
construction avoids this by selecting the minimal degree at which a
non-zero Wronskian coefficient can occur.
\end{rem}
\subsubsection{Control of the remainder term}
Proposition \ref{prop2} provides a uniform lower bound for the
Wronskian of the polynomial part. To transfer this lower bound to the full fibered map, it remains to control the Taylor remainder arising from composition with the monomial curve $\gamma_u$. We begin with a simple uniform continuity lemma.
\begin{lem}\label{lem6}
Let \(f\in C^N(\overline B)\), and suppose that \(\varepsilon_0>0\) has been chosen so that
\[
\|f_\varepsilon-f\|_{C^N(\overline B)}\le 1,
\quad 0<\varepsilon\le\varepsilon_0.
\]
Then there exists \(\delta>0\) such that, for every
\(0<\varepsilon\le \varepsilon_0\), every \(|x|\leq\delta\), and every
multi-index \(\alpha\) with \(|\alpha|\le N\), we have
\[
\left|\partial_\alpha f_\varepsilon(x)
-
\partial_\alpha f_\varepsilon(0)\right|
\le 3.
\]
\end{lem}
\begin{proof}
By the triangle inequality,
\begin{align*}
|\partial_\alpha f_\varepsilon(x)-\partial_\alpha f_\varepsilon(0)|
\le
|\partial_\alpha f_\varepsilon(x)-\partial_\alpha f(x)|
&+
|\partial_\alpha f(x)-\partial_\alpha f(0)|\\
&+
|\partial_\alpha f(0)-\partial_\alpha f_\varepsilon(0)|.
\end{align*}
The first and third terms are bounded by \(1\). Since there are only finitely many multi-indices \(\alpha\) with \(|\alpha|\le N\), and since each \(\partial_\alpha f\) is uniformly continuous on \(\overline B\), we may choose \(\delta>0\) such that the middle term is bounded by \(1\) for all such \(\alpha\) whenever \(|x|<\delta\). Hence the total is bounded by \(3\).
\end{proof}
Now we move on to establish a uniform $C^n$ estimate for the fibered remainder term.
\begin{prop}[Uniform estimate for the fibered remainder]
\label{prop3}
Let \(E_{\varepsilon,u}\) be the fibered remainder term defined in \eqref{14}. Let \(U_0\subset(\mathbb R^\times)^m\) be compact,
and let \(0<t_*\le 1\). There exist constants
\[
\delta>0
\quad\text{and}\quad
C_F>0,
\]
independent of \(0<\varepsilon\le\varepsilon_0\), \(u\in U_0\), and \(t_*\), such that for every $0<\varepsilon\le\varepsilon_0$ and $u\in U_0$, one has
\[
\|E_{\varepsilon,u}\|_{C^n([0,t_*])}
\le
C_F t_*^{Na_1-n},
\]
whenever
\[
|\gamma_u(t)|\le\delta,
\quad 0\le t\le t_*.
\]
\end{prop}
\begin{proof}
For $\vert \beta\vert\leq n$, we write $D^\beta f=\partial_\beta f$. Hence
\[
D^\beta f_{\varepsilon,i}(x)=\sum_{\vert\alpha\vert\leq N-\vert\beta\vert}\frac{\partial_{\alpha+\beta}f_{\varepsilon,i}(0)}{\alpha!}x^\alpha+D^\beta R_{\varepsilon,i}(x).
\]
Differentiating the Taylor expansion with integral remainder gives
\begin{align*}
&D^\beta R_{\varepsilon,i}(x)=\\
&\sum_{\vert \alpha\vert=N-\vert \beta\vert}\frac{N-\vert \beta\vert}{\alpha!}x^\alpha\int_{0}^{1}(1-s)^{N-\vert \beta\vert-1}\cdot [\partial_{\alpha+\beta}f_{\varepsilon,i}(sx)-\partial_{\alpha+\beta}f_{\varepsilon,i}(0)]ds.
\end{align*}
For the multi-indices appearing above, set \(\tau=\alpha+\beta\). There are only finitely many such \(\tau\), and all satisfy \(|\tau|= N\). By Lemma \ref{lem6}, there exists $\delta>0$ such that for every $\vert x\vert\leq \delta$, and every multi-index $\tau$ with $\vert \tau\vert= N$, we have
\[
\vert\partial_\tau f_{\varepsilon,i}(sx)-\partial_\tau f_{\varepsilon,i}(0)\vert \leq 3,
\]
uniformly for $0\leq s\leq 1$ and $0<\varepsilon\leq \varepsilon_0$. Therefore, for $\vert\beta\vert<N$, we have
\begin{align*}
\vert D^\beta R_{\varepsilon,i}(x)\vert&\leq \sum_{\vert\alpha\vert=N-\vert \beta\vert}\frac{N-\vert\beta\vert}{\alpha!}\vert x^\alpha\vert\cdot\int_{0}^{1}3(1-s)^{N-\vert \beta\vert-1}\, ds\\
&=3\sum_{\vert \alpha\vert=N-\vert \beta\vert}\frac{N-\vert \beta\vert}{\alpha!}\vert x^\alpha\vert\cdot \frac{1}{N-\vert \beta\vert}=3\sum_{\vert \alpha\vert=N-\vert \beta\vert}\frac{\vert x^\alpha\vert}{\alpha!}.
\end{align*}
Since $\vert x^\alpha\vert \leq \vert x\vert^{\vert \alpha\vert}=\vert x\vert^{N-\vert\beta\vert}$, we get
\[
\vert D^\beta R_{\varepsilon,i}(x)\vert\leq 3\bigl(\sum_{\vert \alpha\vert= N-\vert\beta\vert}\frac{1}{\alpha!}\bigr)\cdot \vert x\vert^{N-\vert \beta\vert}.
\]
Thus, there exists $C_{N,\beta}>0$ such that 
\begin{align}\label{32}
\vert D^\beta R_{\varepsilon,i}(x)\vert\leq C_{N,\beta}\vert x\vert^{N-\vert \beta\vert}
\end{align}
uniformly for $0<\varepsilon\leq \varepsilon_0$. If \(|\beta|=N\), then by Lemma \ref{lem6}, we have
\[
\vert D^\beta R_{\varepsilon,i}(x)\vert
=
\vert\partial_\beta f_{\varepsilon,i}(x)
-
\partial_\beta f_{\varepsilon,i}(0)\vert\leq 3.
\]
Thus, \eqref{32} holds for all \(|\beta|\le N\) by choosing \(C_{N,\beta}\geq 3\). Now for
\[
\gamma_u(t)=(u_1 t^{a_1},\dots, u_m t^{a_m}),\quad a_j=d^{j-1}+d^m,
\]
we have $a_1=1+d^m<a_2<\cdots<a_m$. Since $U_0$ is compact, there exists
\[
M_{U_0}:=\sup_{u\in U_0}\max_{1\leq j\leq m}\vert u_j\vert<\infty. 
\]
Since $0\leq t\leq t_*\leq 1$ and $a_j\geq a_1$, we have
\[
\vert \gamma_u(t)\vert^2=\sum_{j=1}^{m}u_j^2 t^{2a_j}\leq \sum_{j=1}^{m}M_{U_0}^2\cdot t^{2a_1}.
\]
This means
\begin{align}\label{28}
\vert \gamma_u(t)\vert\leq \sqrt{m}M_{U_0}\cdot t^{a_1}=C_{U_0}t^{a_1}.
\end{align}
Let $1\leq j\leq m$, since $\gamma_{u,j}(t)=u_jt^{a_j}$, we have $\gamma_{u,j}^{(r)}(t)=u_j(a_j)_rt^{a_j-r}$, where $(a_j)_r=a_j(a_j-1)\cdots(a_j-r+1)$. So
\[
\vert \gamma_{u,j}^{(r)}(t)\vert \leq M_{U_0}(a_j)_r\cdot t^{a_1-r}.
\]
Since there are only finitely many $j$, we have
\[
\max_{1\leq j\leq m}(a_j)_r<\infty.
\]
Let $C_r:=M_{U_0}\max_{1\leq j\leq m}(a_j)_r$. Then uniformly for all $u\in U_0$, we have
\begin{align}\label{29}
\vert \gamma_{u,j}^{(r)}(t)\vert \leq C_rt^{a_1-r}.
\end{align}
Recall that $E_{\varepsilon,u,i}(t)=R_{\varepsilon,i}(\gamma_u(t))$.
\begin{itemize}
\item We have
\[
E_{\varepsilon,u,i}'(t)=\frac{d}{dt}R_{\varepsilon,i}(\gamma_u(t))=\sum_{j=1}^{m}\partial_jR_{\varepsilon,i}(\gamma_u(t))\cdot \gamma_{u,j}'(t).
\]
Each term inside the summation has the form $D^\beta R_{\varepsilon,i}(\gamma_u(t))\cdot \gamma_{u,j}'(t)$, where $\vert \beta\vert=1$. More explicitly, $\beta=e_j=(0,\dots,0,1,0,\dots,0)$ where $1$ appears at the $j$-th spot.
\item For the second derivative,
\begin{align*}
E_{\varepsilon,u,i}''(t)=\sum_{j_1,j_2=1}^{m}\partial_{j_1,j_2}&R_{\varepsilon,i}(\gamma_u(t))\cdot \gamma_{u,j_2}'(t)\cdot \gamma_{u,j_1}'(t)+\\
&\sum_{j=1}^{m}\partial_{j}R_{\varepsilon,i}(\gamma_u(t))\cdot \gamma_{u,j}''(t),
\end{align*}
where $\partial_{j_1,j_2}=D^{e_{j_1}+e_{j_2}}$. So each term in the sums takes the form
\[
D^\beta R_{\varepsilon,i}(\gamma_u(t))\cdot \gamma_{u,j_1}^{(r_1)}(t)\cdot \gamma_{u,j_2}^{(r_2)}(t),\quad \vert \beta\vert=2,\quad r_1+r_2=2,
\]
or
\[
D^\beta R_{\varepsilon,i}(\gamma_u(t))\cdot \gamma_{u,j_1}^{(r_1)}(t),\quad \vert \beta\vert=1,\quad r_1=2.
\]
\item In general, for $k$-th derivative, where $0\leq k\leq n$, each term takes the form
\begin{align}\label{30}
D^\beta R_{\varepsilon,i}(\gamma_u(t))\prod_{p=1}^{\vert \beta\vert}\gamma_{u,j_p}^{(r_p)}(t),
\end{align}
where $r_1+\cdots+r_{\vert \beta\vert}=k$ and $0\leq \vert \beta\vert\leq k$. 
\end{itemize}
Now, combine the bounds in (\ref{32}), (\ref{28}) and (\ref{29}) with the fact that, for every $0\leq k\leq n$, there are only finitely many terms in the sum for the derivative, taking form in (\ref{30}). That means, for some $C_E>0$ depending only on \(m\) and \(n\), we have
\begin{align}\label{31}
\vert E_{\varepsilon,u,i}^{(k)}(t)\vert &\leq C_E\cdot \max_{0\leq \vert \beta\vert\leq k}\vert D^\beta R_{\varepsilon,i}(\gamma_u(t))\cdot \prod_{p=1}^{\vert \beta\vert}\gamma_{u,j_p}^{(r_p)}(t)\vert\notag\\
&\leq C_E\cdot \max_{0\leq \vert \beta\vert \leq k}\vert C_{N,\beta}\cdot\vert \gamma_u(t)\vert^{N-\vert\beta\vert}\cdot \prod_{p=1}^{\vert \beta\vert}C_{r_p}t^{a_1-r_p}\vert\notag\\
&\leq C_F\cdot \max_{0\leq \vert \beta\vert\leq k}\vert t^{a_1(N-\vert \beta\vert)}\cdot t^{\sum_{p=1}^{\vert \beta\vert}(a_1-r_p)}\vert\notag\\
&=C_F\cdot \max_{0\leq \vert \beta\vert\leq k}\vert t^{a_1(N-\vert \beta\vert)}\cdot t^{\vert \beta\vert a_1-k}\vert\notag\\
&=C_F\cdot t^{a_1N-k},
\end{align}
for some $C_F>0$, uniformly for $0<\varepsilon\leq \varepsilon_0$ and $u\in U_0$. This directly implies, on $J=[0,t_*]$, we have
\[
\Vert E_{\varepsilon,u}\Vert_{C^n(J)}\leq C_F\cdot t_*^{Na_1-n}.\qedhere
\]
\end{proof}
Proposition \ref{prop3} provides uniform \(C^n\) control of the fibered remainder term. We now use this estimate to bound the difference between the Wronskian of the full fibered map and that of the polynomial part arising from the Taylor expansion of order \(N\).
\begin{prop}[Wronskian comparison]\label{prop4}
Let \(U_0\subset(\mathbb R^\times)^m\) be compact, let $\delta>0$ be as in Proposition~\ref{prop3} and let \(0<t_*\le1\) be such that $\vert\gamma_u(t)\vert\leq \delta$ for all $u\in U_0$ and $0\leq t\leq t_*$. Let
\[
\phi_{\varepsilon,u}(t)
=
P_{\varepsilon,u}(t)+E_{\varepsilon,u}(t)
\]
be the full fibered map defined in \eqref{14}. Then there exist constants \(C_W>0\) and $C_W'>0$, independent of \(0<\varepsilon\le\varepsilon_0\), \(u\in U_0\), and
\(t_*\), such that
\[
\sup_{t\in[0,t_*]}
\left|
W_{\varepsilon,u}(t)-W_{\varepsilon,u}^{P}(t)
\right|
\le
C_W t_*^{\min_{1\leq k\leq n}\left(kNa_1+P_{n-k}-\frac{n(n+1)}{2}\right)}+C_W'\varepsilon,
\]
where \(p_1,\dots,p_n\) denote the pivot exponents constructed in Lemma~\ref{lemp}, and
\begin{align}\label{64}
P_j:=p_1+\cdots+p_j,\quad 1\leq j\leq n.
\end{align}
\end{prop}
\begin{proof}
Recall from (\ref{14}),
\[
P_{\varepsilon,u,i}(t)=\sum_{\vert \alpha\vert\leq N}\Lambda_\alpha^{(i)}(\varepsilon)u^\alpha t^{e_d(\alpha)},
\]
where $e_d(\alpha)=\sum_{j=1}^{m}\alpha_j(d^{j-1}+d^m)$, so $e_d(0)=0$. By definition, $e_d(\alpha)\geq n$ for all $1\leq \vert\alpha\vert\leq N$, since $d\geq d_0\geq N\geq n$. This means, for $1\leq r\leq n$, we have
\begin{align*}
P_{\varepsilon,u}^{(r)}(t)=\sum_{1\leq \vert \alpha\vert\leq N}\Lambda_\alpha(\varepsilon)u^\alpha (e_d(\alpha))_r t^{e_d(\alpha)-r}=\sum_{1\leq \vert \alpha\vert\leq N}c_\alpha(\varepsilon,u)(q_\alpha)_rt^{q_\alpha-r},
\end{align*}
where $c_\alpha(\varepsilon,u)=\Lambda_\alpha(\varepsilon)u^\alpha$, $q_\alpha=e_d(\alpha)$ and $(q)_r=q(q-1)\cdots(q-r+1)$. Clearly, $e_d(\alpha)\geq a_1=1+d^m$ for $\vert \alpha\vert \geq 1$. So $t^{e_d(\alpha)-r}\leq t^{a_1-r}$. Since there are only finitely many $1\leq \vert \alpha\vert\leq N$, $U_0$ is compact and $\Vert f_\varepsilon\Vert_{C^N(\overline{B})}$ is uniformly bounded for $0<\varepsilon\leq \varepsilon_0$. Hence, there exists $C_c>0$ such that
\begin{align}\label{33}
\vert c_{\alpha}(\varepsilon,u)\vert \leq C_c
\end{align}
uniformly for $\varepsilon,u$ and $\vert \alpha\vert \leq N$. Since $d\geq d_0\geq N$, Lemma \ref{lem4} gives injectivity of $e_d$ on all multi-indices satisfying $\vert\alpha\vert\leq N$.

For the Wronskian difference \(W_{\varepsilon,u}-W_{\varepsilon,u}^P\), we have
\[
W_{\varepsilon,u}(t)
=
\det
\begin{pmatrix}
\phi_{\varepsilon,u}(t)\\
\phi'_{\varepsilon,u}(t)\\
\vdots\\
\phi^{(n)}_{\varepsilon,u}(t)
\end{pmatrix}\quad \text{and}\quad W_{\varepsilon,u}^{P}(t)
=
\det
\begin{pmatrix}
P_{\varepsilon,u}(t)\\
P'_{\varepsilon,u}(t)\\
\vdots\\
P^{(n)}_{\varepsilon,u}(t)
\end{pmatrix},
\]
where each row is an $(n+1)$-vector. These are the extended Wronskians. Since
\[
\phi_{\varepsilon,u,0}(t)=P_{\varepsilon,u,0}(t)=1
\]
and
\[
\phi_{\varepsilon,u,0}^{(r)}(t)
=
P_{\varepsilon,u,0}^{(r)}(t)
=
0,
\quad
1\le r\le n,
\]
expanding along the first column shows that both determinants reduce to the corresponding \(n\times n\) Wronskians of the non-constant coordinates. Denote
\[
\widetilde P_{\varepsilon,u}
:=
(P_{\varepsilon,u,1},\dots,P_{\varepsilon,u,n}),
\;
\widetilde E_{\varepsilon,u}
:=
(E_{\varepsilon,u,1},\dots,E_{\varepsilon,u,n}),\; \widetilde\phi_{\varepsilon,u}:=
\widetilde P_{\varepsilon,u}
+
\widetilde E_{\varepsilon,u}.
\]
Therefore,
\[
W_{\varepsilon,u}(t)
=
\det
\begin{pmatrix}
\widetilde\phi_{\varepsilon,u}'(t)\\
\vdots\\
\widetilde\phi_{\varepsilon,u}^{(n)}(t)
\end{pmatrix}
\quad\text{and}\quad
W_{\varepsilon,u}^{P}(t)
=
\det
\begin{pmatrix}
\widetilde P_{\varepsilon,u}'(t)\\
\vdots\\
\widetilde P_{\varepsilon,u}^{(n)}(t)
\end{pmatrix}.
\]

Now, via multi-linearity of the determinant, we get
\[
W_{\varepsilon,u}(t)=\sum_{S\subset\lbrace 1,\dots,n\rbrace}\det(V_1^S(t),\dots,V_n^S(t)),
\]
where
\[
V_r^S(t)=\begin{cases}
\widetilde{E}_{\varepsilon,u}^{(r)}(t), & r\in S\\
\widetilde{P}_{\varepsilon,u}^{(r)}(t), & r\not\in S
\end{cases}.
\]
The term $S=\emptyset$ is exactly $W_{\varepsilon,u}^P(t)$. Hence, we have
\begin{align}\label{34}
W_{\varepsilon,u}(t)-W_{\varepsilon,u}^P(t)=\sum_{\substack{S\subset\lbrace 1,\dots,n\rbrace\\ S\neq \emptyset}}\det (V_1^S(t),\dots, V_n^S(t)).
\end{align}

For any non-empty \(S\subset\{1,\dots,n\}\), let
\[
S^c=\{i_1,\dots,i_h\},
\quad
k=n-h.
\]
Keeping the error rows fixed, we expand each polynomial row using the multi-linearity of the determinant. This yields the exact finite expansion
\begin{align}\label{60}
&\bigl\vert \det(V^S(t))\bigr\vert
=\bigl\vert
\det\bigl(\dots,\widetilde{P}_{\varepsilon,u}^{(i_1)}(t),\dots,\widetilde{P}_{\varepsilon,u}^{(i_h)}(t),\dots\bigr)
\bigr\vert
\notag\\
&=\bigl\vert
\det\bigl(\dots,
\sum_{1\le|\eta|\le N}(q_\eta)_{i_1}\widetilde{c}_\eta(\varepsilon,u)t^{q_\eta-i_1},
\dots,
\sum_{1\le|\eta|\le N}(q_\eta)_{i_h}\widetilde{c}_\eta(\varepsilon,u)t^{q_\eta-i_h},
\dots\bigr)
\bigr\vert
\notag\\
&=\bigl\vert
\sum_{1\le|\eta_1|\le N}\cdots
\sum_{1\le|\eta_h|\le N}
\bigl(
\prod_{a=1}^{h}(q_{\eta_a})_{i_a}
\bigr)
t^{(q_{\eta_1}+\cdots+q_{\eta_h})-(i_1+\cdots+i_h)}
\notag\\
&\qquad\qquad\qquad\qquad\qquad\cdot
\det\bigl(
\dots,
\widetilde{c}_{\eta_1}(\varepsilon,u),
\dots,
\widetilde{c}_{\eta_h}(\varepsilon,u),
\dots
\bigr)
\bigr\vert.\notag\\
&=\bigl\vert \sum_{\substack{(\eta_i)_{i\in S^c}\\ 1\leq \vert \eta_i\vert\leq N}}\bigl(\prod_{i\in S^c}(q_{\eta_i})_i\bigr)t^{\sum_{i\in S^c}q_{\eta_i}-\sum_{i\in S^c}i}\cdot \det\bigl( (\widetilde{E}_{\varepsilon,u}^{(r)}(t))_{r\in S},(\widetilde{c}_{\eta_i}(\varepsilon,u))_{i\in S^c}\bigr)\bigr\vert,
\end{align}
where \(\widetilde c_\eta(\varepsilon,u)\) denotes the vector obtained from \(c_\eta(\varepsilon,u)\) by deleting its first coordinate.

The pivot data of Lemma~\ref{lemp} are those of the unperturbed map $f$, whereas the vectors $\widetilde c_\eta(\varepsilon,u)$ are built from $f_\varepsilon$. Since a perturbation may lower the pivot exponents, see \eqref{58}, we compare with unperturbed coefficient vectors. Write
\[
\widetilde c_\eta(\varepsilon,u)
=
\widetilde c_\eta(0,u)+\delta_\eta(\varepsilon,u),
\quad
\widetilde c_\eta(0,u)
:=
u^\eta\left(
\frac{\partial_\eta f_1(0)}{\eta!},
\dots,
\frac{\partial_\eta f_n(0)}{\eta!}
\right).
\]
Since $\Vert f_\varepsilon-f\Vert_{C^N(\overline{B})}<\varepsilon$ and $U_0$ is compact, there is a constant $C_\delta>0$, independent of $\varepsilon$ and $u$, such that
\begin{align}\label{95}
\vert \delta_\eta(\varepsilon,u)\vert
\leq C_\delta \varepsilon,
\quad
1\leq \vert\eta\vert\leq N,\,
u\in U_0,
\end{align}
moreover, after enlarging $C_c$ in \eqref{33} if necessary, also
$\vert \widetilde c_\eta(0,u)\vert \leq C_c$.
Expanding by multi-linearity in the rows
\[
\widetilde c_{\eta_i}(\varepsilon,u)
=
\widetilde c_{\eta_i}(0,u)
+
\delta_{\eta_i}(\varepsilon,u),
\quad i\in S^c,
\]
and writing $h:=\vert S^c\vert$, the determinant in each summand of \eqref{60} splits as
\[
\det\!\left(
(\widetilde E_{\varepsilon,u}^{(r)}(t))_{r\in S},
(\widetilde c_{\eta_i}(\varepsilon,u))_{i\in S^c}
\right)
=
D^0_{S,\eta}(t)+D^\delta_{S,\eta}(t),
\]
where $D_{S,\eta}^0(t)$ is the single term obtained by choosing
$\widetilde c_{\eta_i}(0,u)$ in every one of the $h$ rows indexed by
$i\in S^c$, while $D_{S,\eta}^\delta(t)$ is the sum of the remaining
$2^h-1$ terms (each of the $h$ rows has two choices), each of which
contains at least one row of the form
$\delta_{\eta_i}(\varepsilon,u)$.

Let
\[
V(Q):=\operatorname{span}\{\widetilde c_\eta(0,u):1\le |\eta|\le N,\ q_\eta\le Q\},
\]
which does not depend on $u\in U_0$, since every $u^\eta$ is non-zero. Define the pivot exponents by
\[
p_j:=\min\{Q:\dim V(Q)\ge j\},\quad 1\le j\le n.
\]
These are precisely the pivot exponents constructed in Lemma~\ref{lemp}, with the initial pivot \(p_0=0\) omitted after passing from the extended \((n+1)\)-dimensional vectors to the last \(n\) coordinates. Since the corresponding spanning conditions are equivalent, the remaining pivot exponents coincide. By \eqref{61}, enlarging the collection of multi-indices from \(|\alpha|\le l\) to \(|\alpha|\le N\) does not affect the pivot exponents. Set
\[
P_j:=p_1+\cdots+p_j,\quad P_0:=0.
\]
If $D_{S,\eta}^0\not\equiv 0$, then the vectors $\widetilde c_{\eta_i}(0,u)$, $i\in S^c$, are linearly independent; in particular, the $\eta_i$ are distinct. Listing the corresponding exponents in increasing order as $q^{(1)}<\cdots<q^{(h)}$, we have 
\[
V(q^{(j)})=\operatorname{span}
\left\{
\widetilde c_\eta(0,u):
1\leq |\eta|\leq N,\ q_\eta\leq q^{(j)}
\right\}.
\]
Hence $\dim V(q^{(j)})\geq j$ and $q^{(j)}\geq p_j$ for $1\leq j\leq h$. Summing over $j$, we obtain
\begin{align}\label{62}
q_{\eta_1}+\cdots+q_{\eta_h}\ge P_h=P_{n-k}\quad\text{whenever}\quad D_{S,\eta}^0\not\equiv 0.
\end{align}
On the other hand, every determinant in $D_{S,\eta}^\delta(t)$ has the rows $\widetilde E_{\varepsilon,u}^{(r)}$, $r\in S$, at least one row of norm at most $C_\delta\varepsilon$, and all its other rows of norm at most $C_c+C_\delta \varepsilon_0$. By the Leibniz formula for the determinant and \eqref{31},
\begin{align}\label{96}
\vert D_{S,\eta}^\delta(t)\vert \ll \varepsilon \prod_{s\in S}t^{Na_1-s}.
\end{align}

By \eqref{31}, together with \eqref{33}, \eqref{62} and \eqref{96}, \eqref{60} becomes
\begin{align}\label{35}
\bigl\vert \det(V^S(t))\bigr\vert
&\ll
\prod_{s\in S}t^{Na_1-s} t^{P_{n-k}-\sum_{i\in S^c}i}+\varepsilon\prod_{s\in S}t^{Na_1-s}\cdot \max_{(\eta_i)_{i\in S^c}}t^{\sum_{i\in S^c}(q_{\eta_i}-i)}\notag\\
&=
t^{kNa_1+P_{n-k}-(\sum_{s\in S}s+\sum_{i\in S^c}i)}+\omega \varepsilon
\notag\\
&\ll
t^{kNa_1+P_{n-k}-\frac{n(n+1)}{2}}+\varepsilon,
\end{align}
where the final line follows from $S\cup S^c=\{1,\dots,n\}$ and 
\[
\omega:=
\prod_{s\in S}t^{Na_1-s}
\cdot
\max_{(\eta_i)_{i\in S^c}}
t^{\sum_{i\in S^c}(q_{\eta_i}-i)}
\leq 1.
\]
Indeed, $0\leq t\leq1$, while $Na_1-s>0$ for $s\in S$ and
$q_{\eta_i}-i\geq a_1-i>0$ for $i\in S^c$. The implied constants of \eqref{35} are independent of $\varepsilon$, $u$ and $t_*$.

Since \(S\neq \emptyset\), we have \(\vert S\vert\geq 1\). For a set \(S\) with \(\vert S\vert=k\), denote
\[
\rho_k:=kNa_1+P_{n-k}-\frac{n(n+1)}{2}.
\]
Since there are only finitely many non-empty subsets \(S\), namely \(2^n-1\), combining \eqref{35} with \eqref{34}, there exist constants \(C_W,C_W'>0\), independent of $\varepsilon$, $u$ and $t_*$, such that
\[
\vert W_{\varepsilon,u}(t)-W_{\varepsilon,u}^P(t)\vert
\leq C_W\max_{1\leq k\leq n} t^{\rho_k}+C_W'\varepsilon
=
C_W t^{\min_{1\leq k\leq n}\rho_k}+C_W'\varepsilon,
\]
where the latter equality follows from $0\leq t\leq t_*\leq 1$. Hence,
\[
\sup_{t\in [0,t_*]}
\vert W_{\varepsilon,u}(t)-W_{\varepsilon,u}^P(t)\vert
\leq
C_W t_*^{\min_{1\leq k\leq n}\rho_k}+C_W'\varepsilon.
\qedhere
\]
\end{proof}
\begin{rem}
For the purpose of establishing the dominance of the polynomial Wronskian, the estimates developed above are optimal within the present approach. In particular, the lower bound for the polynomial Wronskian in Proposition~\ref{prop2} is the strongest consequence of the available pivot information, while the remainder estimates in Proposition~\ref{prop3} and the $\varepsilon$-independent term in the Wronskian comparison in Proposition~\ref{prop4} are likewise sharp with respect to the power of \(t_*\). Consequently, there is no further room to improve the smoothness requirement by refining these estimates alone. We note, however, that the perturbative term in \eqref{35} is estimated crudely by discarding its positive powers of \(t\). Retaining these powers would yield a larger admissible perturbation radius, but would have no effect on the smoothness requirement.

One might wonder whether the constant appearing in Lemma~\ref{lem6}, where the difference of derivatives is bounded by \(3\), is unnecessarily crude. However, replacing this constant by an arbitrarily small parameter \(\eta>0\) only improves the multiplicative constant in the final estimates and does not alter the power of \(t_*\). Although this refinement allows one to relax a strict inequality in the final power comparison to a non-strict one, it does not reduce the required smoothness threshold.
\end{rem}
\subsubsection{Proof of Theorem \ref{upnf}}
Proposition~\ref{prop4} provides a uniform estimate for the difference between the Wronskian of the full fibered map and that of its polynomial part. Combined with the uniform lower bound for the polynomial Wronskian established in Proposition~\ref{prop2}, the proof reduces to comparing the corresponding powers of \(t_*\). We show that, under the smoothness assumption \(N\geq\max\{n,l+1\}\), the contribution of the fibered remainder is of strictly higher-order and is therefore dominated by the polynomial term. Consequently, the full fibered Wronskian inherits the same uniform lower bound, completing the proof of Theorem~\ref{upnf}.
\begin{proof}[Proof of Theorem \ref{upnf}]
Let $d_0=\max\lbrace N,l+1\rbrace$ be as in Lemma~\ref{lemp}, fix a compact set $U_0\subset (\mathbb{R}^\times)^m$ and $d\geq d_0$, and define $a_j$ and $\gamma_u$ as in \eqref{5}. Let \(\varepsilon_0>0\) be chosen as in Lemma \ref{lem3}. All constants below are independent of $0<\varepsilon\leq \varepsilon_0$, $u\in U_0$ and $t_*$. By Proposition \ref{prop2}, there exist constants \(0<t_0\le1\) and \(c_P>0\) such that, for every $0<t_*\leq t_0$, every \(0<\varepsilon\le\varepsilon_0\) and every \(u\in U_0\), we have
\begin{align}\label{36}
\sup_{t\in[0,t_*]}
\bigl|W_{\varepsilon,u}^P(t)\bigr|
\ge c_P t_*^{Q_{\mathrm{good}}},
\end{align}
where \(Q_{\mathrm{good}}\) is given by \eqref{63}. By Proposition \ref{prop4}, there exist constants \(C_W,C_W'>0\),
independent of \(0<\varepsilon\le\varepsilon_0\), \(u\in U_0\), and
\(t_*\), such that, whenever $\vert \gamma_u(t)\vert\leq \delta$ for all $u\in U_0$ and $t\in [0,t_*]$, with $\delta$ as in Proposition~\ref{prop3},
\begin{align}\label{37}
\sup_{t\in[0,t_*]}
\left|
W_{\varepsilon,u}(t)-W_{\varepsilon,u}^{P}(t)
\right|
\le
C_W t_*^{\min_{1\leq k\leq n}\left(kNa_1+P_{n-k}-\frac{n(n+1)}{2}\right)}+C_W'\varepsilon,
\end{align}
where \(P_j\) is defined by \eqref{64} and $a_1=1+d^m$.

We now want the first term on the RHS of \eqref{37} to be dominated by \(\frac{1}{4}c_Pt_*^{Q_{\mathrm{good}}}\); the second term will be controlled by the choice of $\varepsilon_1$. To ensure that this can be achieved by shrinking \(t_*\), it suffices that
\begin{align*}
kNa_1+P_{n-k}-\frac{n(n+1)}{2}>Q_{\text{good}}=P_n-\frac{n(n+1)}{2}
\end{align*}
for every $1\leq k\leq n$. The above inequality can be simplified to
\begin{align}\label{66}
N>\max_{1\leq k\leq n}\frac{P_n-P_{n-k}}{ka_1}=\max_{1\leq k\leq n}\frac{p_{n-k+1}+\cdots +p_n}{ka_1}.
\end{align}
Notice that, for every $1\leq k\leq n$, 
\[
\frac{p_{n-k+1}+\cdots +p_n}{k}\leq p_n,
\]
since the quantity is the average of the largest $k$ pivot exponents constructed in Lemma~\ref{lemp} and $p_1<\cdots<p_n$. Thus the maximum is attained at $k=1$. Hence, \eqref{66} becomes
\begin{align}\label{65}
Na_1>p_n.
\end{align}

Recall that Lemma~\ref{lemp} gives
\[
p_n\leq \max_{1\leq i\leq n}e_d(\alpha_i),\quad\vert\alpha_i\vert\leq l.
\]
For $\vert\alpha\vert\leq l$ and $d\geq l$,
\[
e_d(\alpha)=\vert \alpha\vert d^m+\sum_{j=1}^{m}\alpha_jd^{j-1}\leq ld^m+ld^{m-1}\leq (l+1)d^m.
\]
Therefore, if we want
\[
Na_1=N(1+d^m)\geq (l+1)d^m
\]
to hold in general, we must have $N\geq l+1$; this also makes the above inequality strict. Thus \eqref{65} holds. In conclusion, the sufficient condition on $N$ is
\[
N\geq \max\lbrace n,l+1\rbrace.
\]

Since all the constants above are independent of $t_*$, we may now fix $0<t_*\leq t_0$ so small that $\gamma_u([0,t_*])\subset B$ and $\vert \gamma_u(t)\vert\leq \delta$ for all $u\in U_0$ and $t\in [0,t_*]$, and
\[
C_W t_*^{\min_{1\leq k\leq n}\left(kNa_1+P_{n-k}-\frac{n(n+1)}{2}\right)}<\frac{1}{4}c_Pt_*^{Q_{\text{good}}}.
\]
With $t_*$ fixed, set 
\[
\varepsilon_1:=\min\Bigl\lbrace \varepsilon_0,\frac{c_Pt_*^{Q_{\text{good}}}}{4C_W'}\Bigr\rbrace.
\]
Then \eqref{37} gives, for every $0<\varepsilon\leq \varepsilon_1$ and every $u\in U_0$,
\[
\sup_{t\in [0,t_*]}\vert W_{\varepsilon,u}(t)-W_{\varepsilon,u}^P(t)\vert <\frac{1}{2}c_Pt_*^{Q_{\text{good}}}.
\]
Now fix the interval $[0,t_*]$. Combining the above estimate with (\ref{36}), and using the continuity of both functions on the compact interval $[0,t_*]$, the corresponding suprema are attained. We can then properly choose $t_{\varepsilon,u}\in[0,t_*]$ such that
\[
|W_{\varepsilon,u}^P(t_{\varepsilon,u})|
\ge c_P t_*^{Q_{\mathrm{good}}}\quad\text{and}\quad \vert W_{\varepsilon,u}(t_{\varepsilon,u})-W_{\varepsilon,u}^P(t_{\varepsilon,u})\vert <\frac{1}{2}c_Pt_*^{Q_{\text{good}}}.
\]
By the reverse triangle inequality, we deduce
\begin{align*}
\vert W_{\varepsilon,u}(t_{\varepsilon,u})\vert &\geq \vert W_{\varepsilon,u}^P(t_{\varepsilon,u})\vert-\vert W_{\varepsilon,u}(t_{\varepsilon,u})-W_{\varepsilon,u}^P(t_{\varepsilon,u})\vert\\
&\geq c_Pt_*^{Q_{\text{good}}}-\frac{1}{2}c_Pt_*^{Q_{\text{good}}}=\frac{1}{2}c_Pt_*^{Q_{\text{good}}}.
\end{align*}
Denote
\[
c_*:=\frac{1}{4}c_Pt_*^{Q_{\text{good}}}>0.
\]
Then 
\[
\vert W_{\varepsilon,u}(t_{\varepsilon,u})\vert \geq 2c_*.
\]
Finally, by continuity, there is a small interval $I_{\varepsilon,u}$ around $t_{\varepsilon,u}$ such that
\[
\vert W_{\varepsilon,u}(t)\vert\geq c_*,\quad \forall\, t\in I_{\varepsilon,u}.
\]
The constant $c_*$ is independent of $0<\varepsilon\leq \varepsilon_1$ and $u\in U_0$.
\end{proof}
\begin{rem}
The smoothness threshold \(N\geq\max\{n,l+1\}\) is essentially optimal within the present framework. As observed in the previous remark, replacing the constant \(3\) in Lemma~\ref{lem6} by an arbitrarily small parameter \(\eta>0\) merely relaxes the strict inequality in \eqref{65} to a non-strict one. As the proof shows, this has no effect on the resulting general smoothness requirement.

Philosophically, the additional degree of smoothness is not needed to recover the polynomial structure of the fiber. Rather, it is precisely this extra derivative that provides sufficient control of the remainder term to ensure the dominance of the polynomial part. Since an arbitrary $C^N$ remainder admits only a uniform upper bound, with no corresponding lower bound or cancellation mechanism available, its local behaviour is essentially unrestricted, and the comparison is therefore necessarily one sided. Consequently, the condition $N\geq l+1$ appears intrinsic to the present universal fibering framework, though it may be avoided by allowing the fibering to depend on the manifold.

It is perhaps worth noting that the worst case contribution comes from the mixed determinants \eqref{34} containing exactly one remainder row. Although these terms retain the greatest amount of polynomial structure, they are therefore closest in order to the leading polynomial Wronskian, making them the most difficult to dominate.

Since Theorem~\ref{thmfib} is obtained directly from Theorem~\ref{upnf}, it inherits the same smoothness requirement. Therefore, any further reduction of the smoothness assumption in the fibering theorem would require a genuinely different mechanism for controlling the remainder or a fundamentally different fibering construction. From this perspective, the perturbative nature of the proof should be viewed as a consequence of the present framework rather than its essential feature.
\end{rem}
\subsubsection{Proof of Theorem \ref{thmfib}}
A key feature of Theorem~\ref{upnf} is the uniformity of the construction with respect to both the parameter \(u\in U_0\) and the analytic perturbation \(0<\varepsilon\le\varepsilon_1\). In particular, the compact parameter set \(U_0\), together with the constants \(t_*\) and \(c_*\), are chosen independently of \(\varepsilon\). Consequently, passing from the analytic approximations \(f_\varepsilon\) back to the original \(C^N\) map \(f\) requires no modification of the parameter space, and the family of fibers remains indexed by the same compact set \(U_0\). The proof therefore reduces to an application of the continuity of the Wronskian, with no further geometric identification required.

With Theorem~\ref{upnf} established, it remains to guarantee the existence of analytic manifolds arbitrarily close to a $C^N$ manifold in the \(C^N\) topology. This is provided by the following version of Whitney's analytic approximation theorem; see \citep[Lemma~5]{32813238-652d-3e78-acb7-889d5d49963a}, which is essentially a density statement.
\begin{thm}[Whitney's analytic approximation theorem]\label{whit}
Let $U\subset\mathbb R^m$ be open and let $f:U\to\mathbb R^n$ be of class $C^N$, where $N\in \mathbb{N}$. Let $\overline B\subset U$ be a closed ball. Then for every $\varepsilon>0$ there exists an analytic map $f_\varepsilon:U\to\mathbb R^n$ such that
\[
\|f-f_\varepsilon\|_{C^N(\overline B)}
<
\varepsilon.
\]
\end{thm}
Finally, we reach the main geometric consequence of the paper: \(l\) non-degenerate manifolds satisfying the essentially optimal smoothness threshold \(N\geq\max\{n,l+1\}\) can be fibered into Wronskian non-degenerate curves.
\begin{proof}[Proof of Theorem \ref{thmfib}]
Let $d_0$ be as in Theorem~\ref{upnf}, fix a compact set $U_0\subset (\mathbb{R}^\times)^m$ and $d\geq d_0$, and let $\varepsilon_1>0$, $0<t_*\leq 1$ and $c_*>0$ be the constants given by Theorem~\ref{upnf} for this choice of $U_0$ and $d$. By Theorem \ref{whit}, for every $0<\varepsilon\leq \varepsilon_1$, there exists an analytic map $f_\varepsilon:\overline{B}\to \mathbb{R}^n$ such that
\[
\Vert f-f_\varepsilon\Vert_{C^N(\overline{B})}<\varepsilon.
\]
Applying Theorem~\ref{upnf} to these maps \(f_\varepsilon\), we obtain, for every \(u\in U_0\) and every \(0<\varepsilon\le\varepsilon_1\), 
\begin{align}\label{38}
\sup_{t\in [0,t_*]}\vert W((f_{\varepsilon,1}\circ\gamma_u)',\dots,(f_{\varepsilon,n}\circ\gamma_u)')(t)\vert \geq 2c_*.
\end{align}

Since \(N\ge n\), for every \(0\le k\le n\), each derivative
\[
\frac{d^k}{dt^k}\bigl((f_\varepsilon-f)\circ\gamma_u\bigr)(t)
\]
is a finite sum of terms of the form
\[
D^\beta(f_\varepsilon-f)(\gamma_u(t))
\prod_{p=1}^{|\beta|}
\gamma_{u,j_p}^{(r_p)}(t),
\]
where \(0\le |\beta|\le k\le n\) and $r_1+\cdots+r_{|\beta|}=k$. Since \(\gamma_u([0,t_*])\subset B\subset\overline B\), and by \eqref{29}, all derivatives \(\gamma_{u,j}^{(r)}\), \(0\le r\le n\), are uniformly bounded on \([0,t_*]\), there exists a constant \(C_{U_0}>0\), independent of \(u\in U_0\), such that
\[
\left\|
(f_\varepsilon-f)\circ\gamma_u
\right\|_{C^n([0,t_*])}
\le
C_{U_0}
\|f_\varepsilon-f\|_{C^n(\overline B)}\le
C_{U_0}
\|f_\varepsilon-f\|_{C^N(\overline B)}
\to0,
\]
since $\|f_\varepsilon-f\|_{C^N(\overline B)}\to0$ as $\varepsilon\to0$. Hence
\[
f_\varepsilon\circ\gamma_u\to f\circ\gamma_u
\quad\text{in } C^n([0,t_*]).
\]

Since the determinant is a polynomial expression in its entries, the corresponding Wronskians converge uniformly on \([0,t_*]\).
\[
W((f_{\varepsilon,1}\circ\gamma_u)',\dots,(f_{\varepsilon,n}\circ\gamma_u)')\to W((f_1\circ\gamma_u)',\dots,(f_n\circ\gamma_u)').
\]
Consequently, combining this with \eqref{38}, we obtain
\[
\sup_{t\in [0,t_*]}\vert W((f_{1}\circ\gamma_u)',\dots,(f_{n}\circ\gamma_u)')(t)\vert \geq c_*.
\]
Since the Wronskian is continuous on the compact interval $[0,t_*]$, this supremum is attained. Thus, for every $u\in U_0$, there exists $t_u\in [0,t_*]$ such that
\[
\vert W((f_{1}\circ\gamma_u)',\dots,(f_{n}\circ\gamma_u)')(t_u)\vert \geq c_*>0.
\]
Therefore the curve
\[
\phi_u(t):=f(\gamma_u(t))
\]
is Wronskian non-degenerate at $t_u$.
\end{proof}
\subsubsection{Special case: hypersurfaces}
We now return to the manifold setting. As observed in the remarks following Theorem~\ref{thmfib}, if $f$ is a local parametrisation of an $m$-dimensional manifold, then $\operatorname{rank}Df(0)=m$. In the hypersurface case $m=n-1$, this allows the proof of Theorem~\ref{thmfib} to be sharpened, yielding the optimal smoothness threshold $N\ge\max\{n,l\}$, which is the absolute minimum required to formulate the non-degeneracy assumptions underlying the fibering construction.
\begin{prop}[Hypersurface case]\label{prop:hypersurface}
Assume, in addition, that \(f\) is a local parametrisation of an
\((n-1)\)-dimensional manifold in \(\mathbb R^n\). Then, after a linear
change of coordinates in the parameter space if necessary,
Theorem~\ref{thmfib} remains valid under the weaker smoothness assumption
\[
N\ge\max\{n,l\}.
\]
\end{prop}
\begin{proof}
Since \(f\) is a local parametrisation of an \(m\)-dimensional manifold,
\[
\operatorname{rank}Df(0)=m=n-1.
\]
Let \(e_1,\dots,e_m\) denote the standard basis vectors of
\(\mathbb R^m\), and let
\[
T:=\operatorname{Im}Df(0)
=
\operatorname{span}
\left\{
\partial_{e_1}f(0),\dots,\partial_{e_m}f(0)
\right\}.
\]
Then \(T\) is an \((n-1)\)-dimensional subspace of \(\mathbb R^n\). Since \(f\) is \(l\) non-degenerate at \(0\),
\[
\operatorname{span}
\{\partial_\alpha f(0):1\leq|\alpha|\leq l\}
=
\mathbb R^n.
\]
Therefore, there exists a multi-index \(\kappa\) with 
\begin{align}\label{69}
2\le |\kappa|=:r\le l
\end{align}
such that $\partial_\kappa f(0)\notin T$. Indeed, all first order derivatives belong to $T$, while the derivatives of order at most $l$ span $\mathbb R^n$.

Let $D^rf(0):(\mathbb R^m)^r\rightarrow\mathbb R^n$ denote the symmetric \(r\)-linear map given by
\begin{align}\label{67}
D^rf(0)[v_1,\dots,v_r]
:=
\left.
\frac{\partial^r}{\partial t_1\cdots\partial t_r}
f(t_1v_1+\cdots+t_rv_r)
\right|_{t_1=\cdots=t_r=0}.
\end{align}
Expanding each \(v_j\) with respect to the standard basis and using multi-linearity,
\begin{align}\label{68}
D^rf(0)[v_1,\dots,v_r]
=
\sum_{i_1,\dots,i_r=1}^{m}
\left(
\prod_{j=1}^{r}v_{j,i_j}
\right)
\partial_{e_{i_1}+\cdots+e_{i_r}}f(0).
\end{align}
If $v_1=\cdots=v_r=:v$, then
\[
D^rf(0)[v,\dots,v]=\sum_{i_1,\dots,i_r=1}^{m}v_{i_1}\cdots v_{i_r}\partial_{e_{i_1}+\cdots+e_{i_r}}f(0).
\]

Since \(T\) has codimension \(1\), there exists a non-zero linear functional \(L:\mathbb R^n\to\mathbb R\) such that \(\ker L=T\). By \eqref{69} and \eqref{68}, if $\kappa=(\kappa_1,\dots,\kappa_m)$, then
\[
\partial_\kappa f(0)
=
D^rf(0)
[\underbrace{e_1,\dots,e_1}_{\kappa_1},
\underbrace{e_2,\dots,e_2}_{\kappa_2},
\dots,
\underbrace{e_m,\dots,e_m}_{\kappa_m}].
\]
Since \(\partial_\kappa f(0)\notin T\), we have $L(\partial_\kappa f(0))\neq0$. Hence the composition
\begin{align}\label{70}
L\circ D^rf(0):(\mathbb R^m)^r\to\mathbb R
\end{align}
is not identically zero. Since \(D^rf(0)\) is symmetric and multi-linear and \(L\) is linear, the composition \(L\circ D^rf(0)\) is again a symmetric multi-linear map. 

By \citep[Theorem~1]{THOMAS2014468}, a symmetric multi-linear form is
identically zero if and only if its associated diagonal polynomial is
identically zero. In the present case, the associated diagonal polynomial is
\[
v\mapsto L(D^rf(0)[v,\dots,v]).
\]
By \eqref{70}, there exists some direction \(v\in\mathbb R^m\) such that
\[
L(D^rf(0)[v,\dots,v])\neq0.
\]
Since \(\ker L=T\), it follows that
\[
D^rf(0)[v,\dots,v]\notin T.
\]
Now extend \(v\) to a basis \(v,v_2,\dots,v_m\) of \(\mathbb R^m\). There is a unique linear map $A:\mathbb R^m\to\mathbb R^m$ satisfying
\[
Ae_1=v,\qquad
Ae_j=v_j,\quad 2\le j\le m.
\]
Since \(v,v_2,\dots,v_m\) form a basis of \(\mathbb R^m\), the linear map \(A\) is invertible, that is, $A\in GL_m(\mathbb R)$.

We can then replace \(f\) by the local parametrization
\begin{align}\label{71}
\widetilde{f}(x):=f(Ax).
\end{align}
This is simply a linear change of coordinates in the parameter space. In particular, it does not alter the underlying manifold. Therefore, by \eqref{67} and \eqref{68}, we have
\begin{align*}
D^r\widetilde{f}(0)[e_1,\dots,e_1]=\partial_{re_1}\widetilde{f}(0)&=\frac{\partial^r}{\partial t_1\cdots\partial t_r}\widetilde{f}(t_1e_1+\cdots+t_re_1)\Big|_{t_1=\cdots=t_r=0}\\
&=\frac{\partial^r}{\partial t_1\cdots \partial t_r}f\Bigl(A\bigl(\sum_{i=1}^{r}t_ie_1\bigr) \Bigr)\Big|_{t_1=\cdots=t_r=0}\\
&=\frac{\partial^r}{\partial t_1\cdots \partial t_r}f\Bigl(\sum_{i=1}^{r}t_i v \Bigr)\Big|_{t_1=\cdots=t_r=0}\\
&=D^rf(0)[v,\dots,v]\not\in T.
\end{align*}
By \eqref{71} and the above, we have 
\[
D\widetilde{f}(0)=Df(0)\circ A.
\]

Since \(A\) is invertible, it is bijective. Hence $A(\mathbb R^m)=\mathbb R^m$. Therefore
\[
\mathrm{Im}(Df(0)\circ A)
=
Df(0)(A(\mathbb R^m))
=
Df(0)(\mathbb R^m)
=
\mathrm{Im}Df(0)
=
T.
\]
Thus
\[
T
=
\mathrm{Im}(Df(0)\circ A)
=
\mathrm{Im}(D\widetilde f(0))
=
\mathrm{span}
\left\{
\partial_{e_1}\widetilde f(0),
\dots,
\partial_{e_m}\widetilde f(0)
\right\}.
\]
Hence
\[
\partial_{e_1}\widetilde f(0),
\dots,
\partial_{e_m}\widetilde f(0),
\partial_{re_1}\widetilde f(0)
\]
span \(\mathbb R^n\).

We now revisit the proof of Theorem~\ref{upnf}. Applying the pivot construction used in the proof of Lemma~\ref{lemp} to the spanning family
\[
\partial_{e_1}\widetilde f(0),\dots,
\partial_{e_m}\widetilde f(0),
\partial_{re_1}\widetilde f(0),
\]
we obtain
\begin{align}\label{73}
p_n\le ra_1\le la_1,
\end{align}
since
\[
e_d(e_j)=a_j=d^{j-1}+d^m,\quad
e_d(re_1)=ra_1=r(1+d^m),
\]
and \(r\ge2\) implies \(a_j\le ra_1\) for every \(1\le j\le m\).

The remainder of the proof is identical to that of
Theorem~\ref{upnf}. The only modification is that equality in the estimate
\eqref{65} may now occur, namely
\begin{align}\label{74}
Na_1\ge p_n.
\end{align}
However, as noted after Proposition~\ref{prop4}, the constant \(3\) in
Lemma~\ref{lem6} may be replaced by an arbitrary \(\eta>0\).
Consequently, the same argument yields the required uniform Wronskian
lower bound.

Hence, \eqref{73} together with \eqref{74} imply that the proof goes
through under the weaker assumption \(N\ge l\). Since \(N\ge n\) is
still required in order to form the Wronskian, we conclude that
\[
N\ge\max\{n,l\}.
\qedhere
\]
\end{proof}
\begin{rem}
The above argument relies crucially on the fact that the tangent space has codimension one. After the linear reparametrisation, the unique missing transverse direction is realised by the pure derivative $\partial_{re_1}$, whose weight is $ra_1$. In higher codimension, several independent transverse directions may be required. Although the same reparametrisation argument can still be applied to obtain improved pivot estimates, only one such direction can be assigned to the first coordinate direction, while the remaining directions necessarily involve higher weights. Consequently, the estimate $p_n\le la_1$ is no longer available in general, and the resulting bounds do not improve the smoothness requirement of Theorem~\ref{thmfib}. Although stated for hypersurfaces, the proof only uses the rank condition $\operatorname{rank}Df(0)=n-1$.
\end{rem}
\subsection{Global fibering}
The passage from the local fibering result of Theorem~\ref{thmfib}
to the global statement of Theorem~\ref{thm:globalfib} relies on the
uniform quantitative nature of the construction. Roughly speaking,
the Wronskian lower bound obtained in the proof of
Theorem~\ref{thmfib}, together with Markov's inequality \citep[Theorem~1.1]{pierzchala2016markov}, shows that at
every sufficiently small scale a fixed positive proportion of each
fiber consists of Wronskian non-degenerate points. The Jacobian of the
fibering map then lifts this one-dimensional estimate to a positive
proportion estimate in the parameter space of the manifold. Finally, the Jessen--Marcinkiewicz--Zygmund strong differentiation theorem, in the forms given by \citep[Theorem~1]{kuchta2004martingale} and \citep[Theorem~6]{jessen1935note}, implies that a positive measure exceptional set cannot survive such a
fixed proportional loss at every sufficiently small scale. This
provides the globalisation mechanism underlying
Theorem~\ref{thm:globalfib}.

Throughout this section, we assume the hypotheses of
Theorem~\ref{thmfib} and retain the notation introduced in its proof.
For $u\in U_0$, we write
\[
W_u(t)
:=
W\bigl((f_1\circ\gamma_u)',\ldots,(f_n\circ\gamma_u)'\bigr)(t),
\]
and denote by $W_u^P$ the corresponding polynomial part introduced in
\eqref{14}. Throughout this section we fix $d\geq d_0$, where $d_0$ is as in Theorem~\ref{thmfib}. Lemmas~\ref{lem3}, \ref{lem5} and \ref{lem6} and Propositions~\ref{prop2}, \ref{prop3} and \ref{prop4} remain valid, with the same proofs and constants, for $\varepsilon=0$, that is, with $f_\varepsilon$ replaced by $f$: analyticity of $f_\varepsilon$ is never used, only Taylor's formula of order $N$ with integral remainder, which holds for $C^N$ maps. For $\varepsilon=0$ one has $\delta_\eta\equiv 0$ in the proof of Proposition~\ref{prop4}, so the term $C_W'\varepsilon$ there vanishes. In this section these results are applied with $\varepsilon=0$.
\subsubsection{A positive proportion of good \texorpdfstring{$t$}{t} values}
We first extract from the proof of Theorem~\ref{thmfib} a uniform
positive-proportion estimate for the set of $t$ values at which the
Wronskian does not vanish.
\begin{lem}[Uniform good $t$ proportion]\label{tpro}
There exist $\eta>0$ and $t_1>0$ such that for every $u\in U_0$ and every $0<t_*\leq t_1$, 
\[
\vert \lbrace t\in (0,t_*):W_u(t)\neq 0\rbrace\vert \geq \eta t_*.
\]
\end{lem}
\begin{proof}
For a real polynomial $P$ of degree at most $D$, Markov's inequality \citep[Theorem~1.1]{pierzchala2016markov} on $[-1,1]$ says
\begin{align}\label{75}
\Vert P'\Vert_{\infty,[-1,1]}\leq D^2\Vert P\Vert_{\infty,[-1,1]}.
\end{align}
We rescale the above inequality to $[0,t_*]$. Define
\[
Q(x)=P\left(\frac{t_*}{2}(x+1)\right),\quad x\in [-1,1].
\]
Hence
\[
\Vert Q\Vert_{\infty,[-1,1]}=\Vert P\Vert_{\infty,[0,t_*]}\quad \text{and}\quad \Vert Q'\Vert_{\infty,[-1,1]}=\frac{t_*}{2}\Vert P'\Vert_{\infty,[0,t_*]}.
\]
By \eqref{75}, we then have
\begin{align}\label{76}
\frac{t_*}{2}\Vert P'\Vert_{\infty,[0,t_*]}\leq D^2\Vert P\Vert_{\infty,[0,t_*]}.
\end{align}

Since $W_u^P$ is a polynomial of degree at most $Q_{\max}$, as defined in \eqref{93}, the rescaled version of Markov's polynomial inequality \eqref{76} gives
\begin{align}\label{77}
\sup_{t\in [0,t_*]}\left\vert \frac{d}{dt}W_u^P(t)\right\vert\leq \frac{2Q_{\max}^2}{t_*}\sup_{t\in [0,t_*]}\left\vert W_u^P(t)\right\vert.
\end{align}
Since $W_u^P$ is continuous on $[0,t_*]$, we can choose
$t_u\in[0,t_*]$ such that
\begin{align}\label{78}
\vert W_u^P(t_u)\vert
=
\sup_{t\in[0,t_*]}\vert W_u^P(t)\vert.
\end{align}
For any $t\in[0,t_*]$, the mean value theorem gives
\begin{align}\label{79}
\vert W_u^P(t)-W_u^P(t_u)\vert
\leq
\sup_{s\in[0,t_*]}
\left\vert \frac{d}{ds}W_u^P(s)\right\vert
\vert t-t_u\vert.
\end{align}
Now, using the reverse triangle inequality together with \eqref{77}, \eqref{78}, and \eqref{79}, we obtain
\begin{align*}
\vert W_u^P(t)\vert&\geq \vert W_u^P(t_u)\vert -\vert W_u^P(t)-W_u^P(t_u)\vert\\
&\geq \sup_{s\in [0,t_*]}\vert W_u^P(s)\vert-\frac{2Q_{\max}^2}{t_*}\sup_{s\in [0,t_*]}\vert W_u^P(s)\vert\cdot \vert t-t_u\vert.
\end{align*}
If $\vert t-t_u\vert\leq \frac{t_*}{4Q_{\max}^2}$, then, by the
above and \eqref{94}, we obtain
\begin{align}\label{80}
\vert W_u^P(t)\vert
\geq \frac{1}{2}\sup_{s\in[0,t_*]}\vert W_u^P(s)\vert
\geq \frac{c_Pt_*^{Q_{\text{good}}}}{2}.
\end{align}
Define
\[
J_u(t_*):=
\left[t_u-\frac{t_*}{4Q_{\max}^2},
      t_u+\frac{t_*}{4Q_{\max}^2}\right]\cap[0,t_*].
\]
Then
\[
\vert J_u(t_*)\vert\geq \frac{t_*}{4Q_{\max}^2},
\]
including when $t_u$ is an endpoint. Furthermore, \eqref{80} holds for every $t\in J_u(t_*)$. As shown in Proposition~\ref{prop2} and Proposition~\ref{prop4}, $c_P$ and $C_W$ are independent of $t_*>0$ and $u\in U_0$ once $U_0$ is fixed. Here we take $t_1\leq t_0$ with $t_0$ as in Proposition~\ref{prop2}, so small that $\gamma_u([0,t_1])\subset B$ and $\vert\gamma_u(t)\vert\leq \delta$ for all $u\in U_0$ and $t\in [0,t_1]$, with $\delta$ as in Proposition~\ref{prop3}. Recall that Proposition~\ref{prop4}, applied with $\varepsilon=0$, gives the estimate
\[
\vert W_{u}(t)-W_{u}^P(t)\vert \leq C_Wt^\rho,\quad \rho:=\min_{1\leq k\leq n}\rho_k,
\]
where
\[
\rho_k=kNa_1+P_{n-k}-\frac{n(n+1)}{2}.
\]
The power comparison in the proof of Theorem~\ref{upnf} gives $\rho>Q_{\text{good}}$. Hence, after decreasing $t_1$ if necessary, we have
\begin{align}\label{81}
C_Wt_*^\rho\leq \frac{c_P}{4}t_*^{Q_{\text{good}}},\quad 0<t_*\leq t_1.
\end{align}

For $t\in J_u(t_*)$, combining \eqref{80} and \eqref{81} gives
\[
\vert W_u(t)\vert\geq \frac{c_P}{4}t_*^{Q_{\text{good}}}>0.
\]
Therefore
\[
\vert\lbrace t\in (0,t_*):W_u(t)\neq 0\rbrace\vert\geq \vert J_u(t_*)\vert\geq \frac{t_*}{4Q_{\max}^2}.
\]
Thus the claim holds with $\eta=\frac{1}{4Q_{\max}^2}$.
\end{proof}
\subsubsection{From fiber proportion to \texorpdfstring{$m$}{m}-dimensional proportion}
We now lift the uniform one-dimensional proportion estimate along the
fibers to a corresponding positive proportion estimate in the
$m$-dimensional parameter space.
\begin{lem}[Uniform good parameter space proportion]\label{ppro}
Suppose that $U_0\subset\{1\}\times(\mathbb{R}^{\times})^{m-1}$ is compact and has positive $(m-1)$-dimensional Lebesgue measure, and write $u=(1,u_2,\ldots,u_m)\in U_0$. Let $\gamma_u(t)$ be as defined in \eqref{5}. Define the parameter space fibering map $\Gamma:(0,t_1)\times U_0\to\mathbb{R}^m$ by
\begin{align*}
\Gamma(t,u):=\gamma_u(t)
=
\bigl(t^{a_1},u_2t^{a_2},\ldots,u_mt^{a_m}\bigr).
\end{align*}
For $0<t_*\leq t_1$, let
\begin{align*}
G(t_*):=
\left\{
(t,u)\in(0,t_*)\times U_0:W_u(t)\neq0
\right\}.
\end{align*}
Then there exists $c_0>0$, independent of $t_*$, such that
\[
\vert\Gamma(G(t_*))\vert
\geq
c_0t_*^{a_1+\cdots+a_m}.
\]
\end{lem}
\begin{proof}
For $t>0$, the map $\Gamma$ is injective, since its first coordinate determines $t$, and then each $u_j$ is recovered by dividing the $j$-th coordinate by $t^{a_j}$. Its Jacobian determinant is
\begin{align}\label{82}
\vert \det D\Gamma\vert =\left\vert \det\begin{pmatrix}
a_1t^{a_1-1} & 0 & 0 & \cdots & 0\\
a_2u_2t^{a_2-1} & t^{a_2} & 0 & \cdots & 0\\
a_3u_3t^{a_3-1} & 0 & t^{a_3} & \cdots & 0\\
\vdots & \vdots & \vdots & \ddots & \vdots\\
a_mu_mt^{a_m-1} & 0 & 0 & \cdots & t^{a_m}
\end{pmatrix}\right\vert=a_1t^{a_1+a_2+\cdots+a_m-1}.
\end{align}

For a fixed $u\in U_0$, Lemma~\ref{tpro} gives
\[
\vert E_u\vert:=\vert \lbrace t\in (0,t_*):W_u(t)\neq 0\rbrace\vert
\geq \eta t_*.
\]
Since $t\mapsto t^{a_1+\cdots+a_m-1}$ is increasing, given that the good $t$ proportion has length at least $\eta t_*$, placing the weight as far left as possible achieves the minimum. That is,
\begin{align}\label{83}
\int_{E_u}t^{a_1+\cdots+a_m-1}\,dt \geq \int_{0}^{\eta t_*}t^{a_1+\cdots+a_m-1}\,dt=\frac{(\eta t_*)^{a_1+\cdots+a_m}}{a_1+\cdots+a_m}.
\end{align}

By the area formula \citep[Theorem~3.8]{evans2025measure}, applied with $n=m$, $f=\Gamma$, and $A=G(t_*)$, together with the injectivity of $\Gamma$ for $t>0$, we obtain
\[
\vert\Gamma(G(t_*))\vert
=
\int_{G(t_*)}\vert\det D\Gamma(t,u)\vert\,dt\,du.
\]
By \eqref{82} and \eqref{83}, the above becomes
\begin{align}\label{84}
\vert \Gamma(G(t_*))\vert=a_1\int_{U_0}\int_{E_u}t^{a_1+\cdots+a_m-1}\,dt\,du&\geq a_1\int_{U_0}\frac{(\eta t_*)^{a_1+\cdots+a_m}}{a_1+\cdots+a_m}\,du\notag\\
&=a_1\vert U_0\vert\cdot \frac{(\eta t_*)^{a_1+\cdots+a_m}}{a_1+\cdots+a_m}.
\end{align}
Here $\vert U_0\vert$ denotes the $(m-1)$-dimensional Lebesgue measure
in the hyperplane $\lbrace u_1=1\rbrace$. Therefore the claim follows with
\[
c_0
:=
\frac{a_1\vert U_0\vert}{a_1+\cdots+a_m}
\eta^{a_1+\cdots+a_m}>0,
\]
which is independent of $t_*$.
\end{proof}
\subsubsection{Proof of Theorem~\ref{thm:globalfib}}
We now compare the good parameter proportion obtained in
Lemma~\ref{ppro} with a family of axis parallel rectangles shrinking
to the base point.
\begin{proof}
By compactness of $U_0$ from Lemma~\ref{ppro}, choose
\[
C:=\max\lbrace 1,\sup_{u\in U_0}\vert u_2\vert,\dots,\sup_{u\in U_0}\vert u_m\vert\rbrace.
\]
Then $\Gamma\left((0,t_*)\times U_0\right)\subset R(t_*)$, where
\begin{align}\label{85}
R(t_*):=[-t_*^{a_1},t_*^{a_1}]\times \prod_{j=2}^{m}[-Ct_*^{a_j},Ct_*^{a_j}].
\end{align}
The above has measure
\begin{align}\label{86}
\vert R(t_*)\vert =2^m C^{m-1}t_*^{a_1+\cdots+a_m}.
\end{align}
Combining \eqref{84} and \eqref{86}, we obtain
\begin{align}\label{87}
\vert \Gamma(G(t_*))\vert \geq c\vert R(t_*)\vert,
\end{align}
where
\begin{align*}
c:=\frac{a_1\vert U_0\vert\eta^{a_1+\cdots+a_m}}
{(a_1+\cdots+a_m)2^mC^{m-1}}>0.
\end{align*}
Notice that $c$ is independent of $t_*$. We now use this uniform
positive proportion estimate at arbitrarily small scales to obtain
a density point contradiction, which forms the core of the
globalisation mechanism.

Take a bounded non-empty open subset $V\subset\mathbb{R}^m$ and a
$C^N$ parametrization $f:V\to\mathcal{M}$. Define
\[
\Omega
:=
\left\{
x\in V:
\begin{array}{c}
\exists\, I,\ \phi\in C^N(I,\mathcal M),\ t_x\in I
\text{ such that}\\[2pt]
\phi(t_x)=f(x),\quad
W(\phi_1',\ldots,\phi_n')(t_x)\neq0
\end{array}
\right\},
\]
and let
\[
E:=V\setminus\Omega.
\]
The set $\Omega$ is analytic: indeed, after a suitable reparametrization of the interval, the conditions defining an
admissible pair consisting of a parameter point and a $C^N$ curve
through its image with non-vanishing Wronskian are Borel, and $\Omega$
is the projection of the set of such pairs onto the parameter space.
Hence $\Omega$, and therefore $E$, is Lebesgue measurable
\citep[Theorem~21.10]{kechris2012classical}. The set $E$ therefore consists of those parameter points whose images
do not lie on any $C^N$ curve in $\mathcal M$ with non-vanishing
Wronskian at the corresponding point. Assume, for contradiction, that
\[
|E|>0.
\]

Recall that a point $x\in E$ is called a strong density point of $E$ if 
\[
\frac{\vert E\cap R\vert}{\vert R\vert}\to 1
\]
as axis parallel rectangles $R\ni x$ shrink to $x$. By the Jessen--Marcinkiewicz--Zygmund strong differentiation theorem
\citep[Theorem~1]{kuchta2004martingale}, applied to the characteristic
function $\mathds{1}_E$ of a bounded measurable set, almost every $x\in E$ is a strong density
point of $E$.

By the hypothesis of Theorem~\ref{thm:globalfib} and the definition of
non-degeneracy, $f$ is $l$ non-degenerate at almost every point of
$V$, and the smoothness threshold required in Theorem~\ref{thmfib}
is satisfied. Thus, one may choose $x_0\in E$ such that
\begin{itemize}
\item[1.] $f$ is $l$ non-degenerate at $x_0$;
\item[2.] $x_0$ is a strong density point of $E$.
\end{itemize}
Our model rectangle $R(t_*)$ defined in \eqref{85} is centred at $0$.
Hence, for the translated rectangles $x_0+R(t_*)$, we have
\begin{align}\label{88}
\frac{\vert E\cap (x_0+R(t_*))\vert}
{\vert R(t_*)\vert}\to 1
\end{align}
as $t_*\to0$. Translating the parameter coordinates so that $x_0$ becomes the origin, we may apply the preceding construction at $x_0$.
Consequently, for every sufficiently small $t_*$,
\[
x_0+\Gamma(G(t_*))\subset x_0+R(t_*).
\]
By \eqref{87}, we have
\begin{align}\label{89}
\vert x_0+\Gamma(G(t_*))\vert
\geq c_{x_0}\vert R(t_*)\vert,
\end{align}
where $c_{x_0}>0$ is independent of $t_*$. By definition, we have
\[
\vert E\cap (x_0+\Gamma(G(t_*)))\vert=0.
\]
Hence, by \eqref{89}
\[
\vert E\cap (x_0+R(t_*))\vert\leq \vert R(t_*)\vert-\vert x_0+\Gamma(G(t_*))\vert\leq (1-c_{x_0})\vert R(t_*)\vert.
\]
Therefore
\begin{align}\label{90}
\frac{\vert E\cap (x_0+R(t_*))\vert}{\vert R(t_*)\vert}\leq 1-c_{x_0}
\end{align}
for every sufficiently small $t_*$. Since $c_{x_0}>0$ is fixed, \eqref{90} contradicts \eqref{88}. Thus $\vert E\vert=0$. Finally, $f$ is locally Lipschitz, so $f(E)$ has zero $m$-dimensional Hausdorff measure, and every point of $\mathcal{M}\setminus f(E)\subset f(\Omega)$ has the property stated in Theorem~\ref{thm:globalfib}.
\end{proof}
\section{Applications to Diophantine approximation}
Theorem~\ref{upnf} and Theorem~\ref{thmfib} provide a local uniform quantitative Wronskian fibering mechanism for \(C^N\) manifolds under the essentially optimal smoothness threshold \(N\geq\max\{n,l+1\}\). Furthermore, Theorem~\ref{thm:globalfib} provides a global fibering mechanism, showing that almost every point of the manifold lies on a curve which is Wronskian non-degenerate at that point. Consequently, existing arguments based on fibering, particularly those concerning
dimension theoretic and measure theoretic results, previously restricted to the analytic setting or requiring relatively high smoothness, may now be revisited under essentially optimal finite smoothness assumptions.
\subsection{Badly approximable points on manifolds}
The notion of badly approximable numbers originates from Dirichlet’s theorem; see \citep[Theorem 1.2]{beresnevich2016metricdiophantineapproximationaspects}. A real number $x$ is called badly approximable if there exists a constant $c=c(x)>0$ such that
\[
\Big\vert x-\frac{p}{q}\Big\vert \geq \frac{c}{q^2}
\]
for all $(p,q)\in \mathbb{Z}\times\mathbb{N}$. The set of badly approximable real numbers will be denoted by
\[
\mathbf{Bad}:=\lbrace x\in \mathbb{R}\setminus \mathbb{Q}:\inf_{q\in \mathbb{N}}q\Vert qx\Vert>0\rbrace,
\]
where $\Vert x\Vert$ denotes the distance of $x$ from the nearest integer. By Khintchine's theorem \citep{Khintchine1924}, $\mathbf{Bad}$ has Lebesgue measure zero, while Jarník \citep{Ja1928} proved that it nevertheless has full Hausdorff dimension,
\[
\dim \mathbf{Bad}=1.
\]

The higher-dimensional analogue of badly approximable numbers is motivated by the multidimensional version of Dirichlet’s theorem, which follows from Minkowski’s linear forms theorem; see \citep[Theorem 4.4]{beresnevich2016metricdiophantineapproximationaspects}. Given a weight vector $r=(r_1,\dots,r_n)\in\mathcal R_n$, where
\[
\mathcal R_n=\left\{r\in\mathbb R^n:r_i\ge0,\ \sum_{i=1}^nr_i=1\right\},
\]
the corresponding set of $r$-badly approximable points is defined by
\begin{align*}
\mathbf{Bad}(r)=\lbrace x=(x_1,\dots,x_n)&\in \mathbb{R}^n:\,\exists\,c=c(x)>0\text{ such that}\\
& \max_{1\leq i\leq n}\Vert qx_i\Vert^{1/r_i}\geq cq^{-1},\,\forall q\in \mathbb{N}\rbrace.
\end{align*}
Schmidt \citep{Schmidt1966ONBA} introduced his eponymous game and proved that the symmetric set $\mathbf{Bad}(1/n,\ldots,1/n)$ is winning, a property which in particular implies full Hausdorff dimension. More generally, one is interested in the intersection of weighted badly approximable points with manifolds. Some non-degeneracy assumption is necessary for meaningful dimension statements, since degenerate manifolds may lie in rational affine subspaces whose intersection with $\mathbf{Bad}(r)$ is empty.

Using Cantor-type constructions \citep{d9b46a81-48db-38e7-aded-22cc3c327ebe} together with estimates for dangerous intervals, Beresnevich \citep[Theorem~2]{Beresnevich_2015} established the following full Hausdorff dimension result for weighted badly approximable points on non-degenerate curves.
\begin{thm}\label{thm6}
Let $n\in \mathbb{N}$, $n\geq 2$, $I\subset \mathbb{R}$ be an open interval and $\mathcal{F}_n(I)$ be a finite family of maps defined on $I$ Wronskian non-degenerate at the same point $x_0\in I$. Let $W$ be a finite or countable subset of $\mathcal{R}_n$ such that
\begin{align}\label{1}
\inf\lbrace \tau(r):=\min\lbrace r_i:\,r_i\neq 0\rbrace:\,r\in W\rbrace>0.
\end{align}
Then
\[
\dim\bigcap_{f\in \mathcal{F}_n(I)}\bigcap_{r\in W}f^{-1}(\mathbf{Bad}(r))=1.
\]
\end{thm}
Subsequently, Yang \citep[Theorem 1.8]{yang2019badlyapproximablepointsmanifolds} established the corresponding Hausdorff dimension result for arbitrary $C^n$ non-degenerate manifolds, simultaneously removing both the analyticity assumption and the restriction \eqref{1} on the weights. In contrast to Beresnevich's approach, Yang's proof relies on homogeneous dynamics and the linearization technique rather than the fibering method. Thus, while the Hausdorff dimension problem has since been resolved for $C^n$ non-degenerate manifolds, the analytic fibering mechanism itself remained unavailable beyond the analytic category.

To lift the curve level results to manifolds, we use the following version of Marstrand's slicing lemma; see \citep[Theorem 10.11]{Mattila_1995}.
\begin{lem}[Marstrand's slicing lemma]\label{lem1}
Let $m>1$, $S\subset\mathbb{R}^m$, $s>0$, $U\subset\mathbb{R}^{m-1}$ such that $\dim \lbrace (t,u_2,\dots,u_m)\in S \rbrace\geq s$ for each $(u_2,\dots, u_m)\in U$. Then
\[
\dim S\geq \dim U+s.
\]
\end{lem} 
Combining the analytic fibering lemma \ref{lem2} with Marstrand's slicing lemma \ref{lem1}, Beresnevich \citep[Theorem~1]{Beresnevich_2015} extended Theorem~\ref{thm6} from curves to analytic non-degenerate manifolds, obtaining the corresponding full Hausdorff dimension result under the restriction \eqref{1}.

Replacing the analytic fibering lemma \ref{lem2} by Theorem~\ref{thmfib}, and using Marstrand's slicing lemma \ref{lem1} together with Yang's result for non-degenerate curves \citep[Theorem~1.8]{yang2019badlyapproximablepointsmanifolds}, yields the following finite smoothness analogue. In particular, Yang's result allows us to remove the restriction \eqref{1} on the weights.
\begin{thm}[Badly approximable points on smooth manifolds]\label{bad}
Let $m\le n$ be positive integers, and let $B$ be an open ball in $\mathbb R^m$ centred at $0$. Let $f:\overline{B}\to\mathbb R^n$ be a $C^N$ map, where $N$ satisfies \eqref{45}. Suppose there exist non-zero multi-indices $\alpha_1,\dots,\alpha_n$ such that \eqref{46} holds. Let $W$ be a finite or countable subset of $\mathcal{R}_n$. Then
\[
\dim \bigcap_{r\in W}f^{-1}(\mathbf{Bad}(r))=m.
\]  
\end{thm}
In particular, this extends \citep[Theorem~1]{Beresnevich_2015} to the finite smoothness setting within the fibering framework.
\begin{rem}
Theorem~\ref{bad} is stated for a single manifold only for clarity. The same argument applies to any finite collection of sufficiently smooth manifolds by considering a finite product of Wronskians. No new ideas are required, although this may lead to a higher smoothness requirement.
\end{rem}
\begin{proof}[Proof of Theorem~\ref{bad}]
The case $m=1$ is immediate, so assume henceforth that $m\geq2$ and fix $u_1=1$. Let $d_0$ be as in Theorem~\ref{thmfib}, fix $d\geq d_0$ and fix a compact set $U_0\subset\lbrace 1\rbrace\times(\mathbb{R}^\times)^{m-1}\subset (\mathbb R^\times)^m$ with non-empty relative interior. By Theorem~\ref{thmfib}, there exists $0<t_*\leq 1$ such that 
\[
\gamma_u([0,t_*])\subset B,
\]
for every $u=(u_1,\dots,u_m)\in U_0$, where $\gamma_u$ is defined in \eqref{5}, and such that, for every
$u\in U_0$, the fiber
\[
\phi_u(t):=f(\gamma_u(t))
\]
is Wronskian non-degenerate at some point
$t_u\in(0,t_*)$.

Fix $u_0$ in the relative interior of $U_0$. By the joint continuity of the Wronskian in
$(u,t)$, there exist a neighbourhood $U\subset U_0$ of $u_0$ and
an open interval $I$ with $\overline{I}\subset(0,t_*)$ containing $t_{u_0}$ such that
$\phi_u$ is Wronskian non-degenerate on $I$ for every $u\in U$.
Define
\[
D:=
\bigl\{
(t,u_2,\dots,u_m)\in\mathbb R^m:
t\in I,\;
u=(u_1,\dots,u_m)\in U
\bigr\}.
\]
Thus, \citep[Theorem~1.8]{yang2019badlyapproximablepointsmanifolds} is applicable and
\[
S_u:=I\cap\bigcap_{r\in W}\phi_u^{-1}(\mathbf{Bad}(r))
\]
has Hausdorff dimension $1$ for every $u\in U$. Denote
\[
S=\lbrace (t,u_2,\dots,u_m):\,t\in S_u,\ u\in U\rbrace\subset D.
\]
Since $\dim S_u=1$ for every $u\in U$, Marstrand's slicing lemma
\ref{lem1} implies that $\dim S\geq \dim U+1$. Since $U$ is a
neighbourhood in $(\mathbb R^\times)^{m-1}$, we have
$\dim U=m-1$, and therefore $\dim S\geq m$. Let $S'\subset B$ be the image of $S$ under the map
\begin{align}\label{39}
(t,u_2,\dots,u_m)\mapsto (t^{1+d^m},u_2t^{d+d^m},u_3t^{d^2+d^m},\dots,u_mt^{d^{m-1}+d^m}).
\end{align}
By construction,
\[
S'\subset \bigcap_{r\in W}f^{-1}(\mathbf{Bad}(r)).
\]
The map \eqref{39} is injective on $D$ and both it and its inverse (on its image) have bounded continuous derivatives. Hence it is bi-Lipschitz on $D$. It is well known that bi-Lipschitz maps preserve Hausdorff dimension, see for example \citep[Corollary 2.4]{falconer2013fractal}. Therefore, $\dim S'=\dim S\geq m$. Since
\[
S'
\subset
\bigcap_{r\in W}f^{-1}(\mathbf{Bad}(r))
\subset B\subset\mathbb R^m,
\]
we obtain
\[
m
\le
\dim\bigcap_{r\in W}f^{-1}(\mathbf{Bad}(r))
\le m.
\]
Therefore
\[
\dim\bigcap_{r\in W}f^{-1}(\mathbf{Bad}(r))
=
m.\qedhere
\]
\end{proof}
\subsection{The convergent Gallagher property}
During the preparation of this work, we became aware of a recent fibering lemma due to Chow, Srivastava, Technau and Yu \citep[Lemma A.1]{chow2026rationalpointshyperbolicregions}, which was used to establish the convergent Gallagher property for smooth non-degenerate manifolds. We briefly compare their fibering mechanism with the global fibering mechanism provided by Theorem~\ref{thm:globalfib}.

The two constructions are fundamentally different. The fibering lemma
in \citep[Lemma~A.1]{chow2026rationalpointshyperbolicregions} is a
direct construction producing curves that are \(L\) non-degenerate
for some \(L\) depending on the manifold dimension \(m\) and the
original order of non-degeneracy \(l\). By
\citep[Proposition~2.13]{beresnevich2023khintchine}, these curves are
Wronskian non-degenerate away from at most a countable set of points.
In contrast, Theorem~\ref{thm:globalfib} shows, under the essentially
optimal smoothness threshold \(N\geq\max\{n,l+1\}\), that almost every
point of the manifold lies on a curve which is Wronskian
non-degenerate at that point. This provides the global fibering mechanism required for the measure theoretic application considered in this section and allows the resulting Gallagher theorem to be formulated on a bounded parameter domain without imposing compactness.

The smoothness requirements of the two fibering mechanisms also differ substantially. Translating the assumptions of \citep[Lemma A.1]{chow2026rationalpointshyperbolicregions} into our notation, the resulting order \(L\) may be as large as \(l^{m+1}\), with necessarily \(L\geq n\). Consequently, their reduction may require smoothness of order comparable to
\[
\max\{n,l^{m+1}\}.
\]
In contrast, Theorem~\ref{thm:globalfib} requires only
\[
N\geq\max\{n,l+1\}.
\]
In particular, the smoothness threshold can be significantly weaker when the manifold dimension \(m\) is large. This comparison concerns the fibering step only; the curve theorem \citep[Theorem~1.14]{chow2026rationalpointshyperbolicregions} is itself proved for $C^\infty$ curves.

Furthermore, Proposition~\ref{prop:smoothness-comparison} shows that, for the same collection of multi-indices witnessing non-degeneracy, our smoothness requirement never exceeds that arising from \citep[Lemma A.1]{chow2026rationalpointshyperbolicregions}.
\begin{prop}\label{prop:smoothness-comparison}
Let \(N=\max\{n,l+1\}\) be the smoothness requirement in Theorem~\ref{thm:globalfib}, and let \(L\) denote the order arising from \citep[Lemma A.1]{chow2026rationalpointshyperbolicregions} for the same collection of multi-indices witnessing non-degeneracy. Then
\[
N\le L.
\]
\end{prop}
\begin{proof}
Let $\alpha_1,\dots,\alpha_n\in\mathbb{Z}_{\ge0}^m$ with
\[
\operatorname{span}
\left\{
\partial_{\alpha_1}f(0),\dots,
\partial_{\alpha_n}f(0)
\right\}
=\mathbb{R}^n,
\]
where $1\le |\alpha_i|\le l$. Set
\[
r:=\max_{1\le i\le n}|\alpha_i|\le l,
\quad
d=l+2.
\]
By definition,
\[
L=\max_{1\leq i\leq n} e_d(\alpha_i)\quad\text{where}\quad
e_d(\alpha)
=
\sum_{j=1}^{m}\alpha_j(d^{j-1}+d^m).
\]
Choose $\alpha_{i_0}$ with $|\alpha_{i_0}|=r$. Since $d^{j-1}\ge1$,
\[
L
\ge
e_d(\alpha_{i_0})
\ge
r(1+d^m).
\]
As $r\ge1$,
\[
L
\ge
1+d^m
\ge
1+d
=
l+3
>
l+1.
\]
Moreover, by construction, the resulting curve is \(L\) non-degenerate in \(\mathbb R^n\), and therefore necessarily \(L\ge n\). Hence
\[
L
\ge
\max\{n,l+1\}
=
N,
\]
as required.
\end{proof}
We begin by recalling the notion of multiplicative approximation. We say $x\in \mathbb{R}^n$ is $\psi$-multiplicatively approximable if 
\[
\prod_{i\leq n}\Vert qx_i\Vert<\psi(q)
\]
holds for infinitely many $q\in \mathbb{N}$. Let $\mathcal{W}_n^\times (\psi)$ be the set of $\psi$-multiplicatively approximable vectors in $\mathbb{R}^n$. A Borel measure \(\mu\) on \(\mathbb R^n\) is convergent Gallagher if, for any non-increasing function $\psi:\mathbb{N}\to [0,1)$,
\[
\sum_{q=1}^\infty \psi(q)(\log q)^{n-1}<\infty\quad\Rightarrow\quad \mu(\mathcal{W}_n^\times(\psi))=0.
\]
For an integer $1\leq m\leq n$, let $V\subset\mathbb{R}^m$ be a bounded
open set. As usual, the primary object of interest is the manifold
\[
\mathcal{M}=\lbrace f(x):\,x\in V\rbrace,
\]
parametrised by a sufficiently smooth map
$f=(f_1,\dots,f_n):V\to\mathbb{R}^n$. The manifold theorem is proved by reduction to curves. Let $I\subset \mathbb{R}$ be a closed interval contained in the open interval $U\subset \mathbb{R}$. Consider the curve
\begin{align}\label{40}
\mathcal{C}=\lbrace f(x):\, x\in I\rbrace,
\end{align}
where $f=(f_1,\dots,f_n): U\to \mathbb{R}^n$ is of class $C^n$. The reduction relies on the following curve theorem, which is the case of curves of \citep[Theorem~1.2]{beresnevich2026weighteddiophantineapproximationmanifolds}.
\begin{thm}\label{chowc}
If the curve $\mathcal{C}$ from (\ref{40}) is Wronskian non-degenerate, then the natural Lebesgue measure on $\mathcal{C}$ is convergent Gallagher.
\end{thm}
Combining Theorem~\ref{chowc} with the global fibering mechanism
underlying Theorem~\ref{thm:globalfib} gives a fibering proof of the
following result. It is a special case of \citep[Theorem~1.2]{beresnevich2026weighteddiophantineapproximationmanifolds}, which needs no smoothness beyond that implicit in non-degeneracy. We include the argument because it shows how the global fibering mechanism combines with a curve level result; the proof of Theorem~\ref{thm:twisted} follows the same structure.
\begin{thm}[Convergent Gallagher property]\label{thmcgp}
Let 
\[
\mathcal{M}=\lbrace f(x):x\in V\rbrace\subset \mathbb{R}^n,
\]
be a $C^N$ manifold of dimension $m\geq 1$, parametrised on a bounded domain $V\subset \mathbb{R}^m$. Assume that $f$ is $l$ non-degenerate at almost every $x\in V$ and the smoothness condition \eqref{45} is satisfied. Then the natural Lebesgue measure on $\mathcal{M}$ is convergent Gallagher.
\end{thm}
\begin{proof}
Let $\psi:\mathbb{N}\to[0,1)$ be non-increasing and suppose
\[
\sum_{q\geq1}\psi(q)(\log q)^{n-1}<\infty.
\]
Set
\[
\mathcal E_\psi
:=
\{x\in V:f(x)\in\mathcal W_n^\times(\psi)\}.
\]
Suppose, for a contradiction, that $|\mathcal E_\psi|>0$. Since $f$
is $l$ non-degenerate almost everywhere, we may choose
$x_0\in\mathcal E_\psi$ which is both a point of $l$ non-degeneracy
of $f$ and a strong density point of $\mathcal E_\psi$.

Applying the global fibering construction from the proof of
Theorem~\ref{thm:globalfib} at $x_0$, that is, to the translated map $x\mapsto f(x_0+x)$, let $U_0$, $\Gamma(t,u)$ and $G(t_*)$ be as defined in Lemma~\ref{ppro}, and let $R(t_*)$ be as defined in \eqref{85}; accordingly, $W_u$ below denotes the Wronskian of $t\mapsto f(x_0+\gamma_u(t))$. For every sufficiently small $t_*>0$ we obtain
\begin{align}\label{91}
x_0+\Gamma(G(t_*))\subset x_0+R(t_*)
\end{align}
and recall from \eqref{87} that
\begin{align}\label{92}
|\Gamma(G(t_*))|\geq c|R(t_*)|
\end{align}
for some $c>0$ independent of $t_*$.

For each $u\in U_0$, let
\[
J_u^*(t_*):=\{t\in(0,t_*):W_u(t)\neq0\}.
\]
Since $W_u$ is continuous, $J_u^*(t_*)$ is open and hence is a
countable union of disjoint open intervals. On every compact
subinterval of each component, the corresponding curve
\[
\phi_u(t):=f(x_0+\gamma_u(t))
\]
is Wronskian non-degenerate. Since Theorem~\ref{chowc} is stated for closed intervals, we exhaust each component of $J_u^*(t_*)$ by a countable increasing union of compact subintervals. Applying Theorem~\ref{chowc} on each such
subinterval, and using countable sub-additivity, gives
\[
\left|
\left\{
t\in J_u^*(t_*):
\phi_u(t)\in\mathcal W_n^\times(\psi)
\right\}
\right|=0
\]
for every $u\in U_0$.

Define
\[
Z(t_*):=
\left\{
(t,u)\in G(t_*):
f(x_0+\gamma_u(t))\in\mathcal W_n^\times(\psi)
\right\}.
\]
By Fubini's theorem,
\[
|Z(t_*)|=0.
\]
Applying the area formula \citep[Theorem~3.8]{evans2025measure} for $\Gamma$, together with its injectivity
for $t>0$, gives
\[
|\Gamma(Z(t_*))|
=
\int_{Z(t_*)}
|\det D\Gamma(t,u)|\,dt\,du
=0.
\]
Since translation by $x_0$ preserves Lebesgue measure, $|x_0+\Gamma(Z(t_*))|=0$. Moreover, by the definition of $\mathcal E_\psi$,
\[
x_0+\Gamma(Z(t_*))
=
\mathcal E_\psi
\cap
\bigl(x_0+\Gamma(G(t_*))\bigr).
\]
Therefore
\[
\left|
\mathcal E_\psi
\cap
\bigl(x_0+\Gamma(G(t_*))\bigr)
\right|
=0.
\]
Consequently, by \eqref{91} and \eqref{92}, we have
\begin{align*}
\left|
\mathcal E_\psi\cap
\bigl(x_0+R(t_*)\bigr)
\right|
&=
\left|
\mathcal E_\psi\cap
\Bigl(
\bigl(x_0+R(t_*)\bigr)
\setminus
\bigl(x_0+\Gamma(G(t_*))\bigr)
\Bigr)
\right|\\
&\leq
\left|
\bigl(x_0+R(t_*)\bigr)
\setminus
\bigl(x_0+\Gamma(G(t_*))\bigr)
\right|\\
&=
|R(t_*)|-|\Gamma(G(t_*))|\\
&\leq
(1-c)|R(t_*)|.
\end{align*}
Thus
\[
\frac{
\left|\mathcal E_\psi\cap\bigl(x_0+R(t_*)\bigr)\right|
}{
|R(t_*)|
}
\leq 1-c
\]
for every sufficiently small $t_*>0$, contradicting the fact that
$x_0$ is a strong density point of $\mathcal E_\psi$. Hence
\[
|\mathcal E_\psi|=0.
\]

Since $f$ is a local parametrisation of $\mathcal M$, it follows that
\[
\mu_m\bigl(\mathcal M\cap\mathcal W_n^\times(\psi)\bigr)=0.
\]
Therefore the natural Lebesgue measure on $\mathcal M$ is convergent
Gallagher.
\end{proof}
\begin{rem}
The proof above illustrates how the global fibering mechanism can be
combined with a curve level Diophantine result. In the proof of
Theorem~\ref{thm:globalfib}, the exceptional set misses the good
fibering region directly from its definition. Here this does not
follow from the definition of $\mathcal E_\psi$. Instead,
Theorem~\ref{chowc}, together with Fubini's theorem and the area
formula, shows that $\mathcal E_\psi$ has measure zero in the good
region. This reduces the argument to the same density mechanism:
at every sufficiently small scale the exceptional set misses a fixed
positive proportion of the parameter space, and the strong density
theorem then gives the global measure-zero conclusion.
\end{rem}
\subsection{Twisted Diophantine approximation on manifolds}
We conclude with an application of Theorem~\ref{thm:globalfib} to twisted Diophantine approximation on manifolds. Fix an $m\times n$ real matrix $\alpha$ and an approximation function,
that is, a non-increasing continuous map $\psi:\mathbb{R}_+\to\mathbb{R}_+$. A vector $\beta\in\mathbb R^m$ is called $(\psi,\alpha)$-approximable if there exist infinitely many pairs $(p,q)\in\mathbb Z^m\times(\mathbb Z^n\setminus\{0\})$ such that
\[
|\alpha q+\beta+p|\leq\psi(|q|).
\]
The set of such vectors is denoted by $\mathcal T_\psi(\alpha)$. Twisted Diophantine approximation \citep{beresnevich2008classical,bugeaud2010shrinking} concerns the metric properties of $\mathcal T_\psi(\alpha)$, where $\alpha$ is fixed while $\beta$ is allowed to vary, possibly subject to the restriction that $\beta$ lies on a manifold.

For a fixed matrix $\alpha\in\mathbb R^{m\times n}$, a vector
$\beta\in\mathbb R^m$ is called badly $\alpha$-approximable if there
exists a constant $c>0$ such that
\[
|\alpha q+\beta+p|
\geq
c|q|^{-n/m}
\]
for every $(p,q)\in\mathbb Z^m\times(\mathbb Z^n\setminus\{0\})$. Equivalently, $\beta$ does not belong to $\mathcal T_\psi(\alpha)$ for the approximation function $\psi(q)=cq^{-n/m}$. We denote the set of badly $\alpha$-approximable vectors by $\mathcal{T}\mathbf{ad}(\alpha)$. The set $\mathcal{T}\mathbf{ad}(\alpha)$ is known to have full Hausdorff dimension \citep{broderick2013badly,bugeaud2010shrinking}, while its Lebesgue measure depends on the Diophantine properties of $\alpha$.

In \citep{beresnevich2026twisteddiophantineapproximationmanifolds},
Beresnevich, Simmons and Velani established Khintchine type convergence and divergence results for $\mathcal T_\psi(\alpha)$, together with a zero measure result for $\mathcal{T}\mathbf{ad}(\alpha)$, on non-degenerate analytic manifolds. As observed in \citep[Remarks~1.3 and~2.1]{beresnevich2026twisteddiophantineapproximationmanifolds}, analyticity is used only in the passage from curves to manifolds via the analytic fibering lemma \ref{lem2} and Fubini’s theorem. Since Theorem~\ref{thm:globalfib} provides the required global fibering mechanism, following the proof structure of Theorem~\ref{thmcgp} yields the following extension. Here, $m$ denotes the ambient dimension in the notation of \citep{beresnevich2026twisteddiophantineapproximationmanifolds}.
\begin{thm}\label{thm:twisted}
Let $m\ge2$. The conclusions of
\citep[Theorems~2--4]{beresnevich2026twisteddiophantineapproximationmanifolds}
remain valid if every occurrence of a non-degenerate analytic manifold is replaced by a $C^N$ non-degenerate manifold satisfying the hypotheses of Theorem~\ref{thm:globalfib}.
\end{thm}
\section{Final remarks}
In this final section, we briefly discuss several future directions of the fibering mechanism developed in this paper. One natural direction concerns Schmidt type games and winning properties. In \citep{beresnevich2020winningpropertybadlyapproximable}, Beresnevich, Nesharim and Yang established an absolute winning result for weighted badly approximable points on non-degenerate analytic curves. The corresponding problem for higher-dimensional manifolds remains open. The fibering mechanism presented here provides a natural route towards this question. However, unlike Hausdorff dimension and measure theoretic statements, winning properties cannot currently be recovered from fibering via classical slicing arguments such as Marstrand's slicing lemma~\ref{lem1} or Fubini's theorem. This suggests that an appropriate lifting principle for winning sets is required. A possible approach is to work in the parameter space, where product theorems for winning sets can be studied, and investigate whether the winning property can be preserved under suitable parametrisations of manifolds. 

Another related direction is to remove the analyticity assumption from the results of \citep{beresnevich2020winningpropertybadlyapproximable}. In that work, analyticity enters through the quantitative non-divergence estimates of Kleinbock and Margulis \citep{kleinbock1998flowshomogeneousspacesdiophantine}. Eliminating this assumption appears to require new ideas. One possible approach is to perturb a sufficiently smooth manifold to a nearby analytic manifold and carry out the argument entirely within the analytic category. The uniform perturbative non-degeneracy established here may provide a useful framework for investigating such an approach.

A second direction concerns weakly non-degenerate manifolds; see
\citep{BERESNEVICH20103064}. In the planar setting considered in
\citep{BERESNEVICH20103064}, weak non-degeneracy is formulated through
local $C^0$ approximation by $C^2$ functions whose curvature is bounded
above and below by fixed positive constants. In this way, the
approximating curves provide a surrogate for the derivative information
ordinarily used to rule out degeneracy, while retaining quantitative
geometric rigidity. Although the smoothness assumptions required for
the fibering theorem are essentially optimal, the arithmetic
applications themselves may admit greater flexibility. The uniform
Wronskian lower bounds obtained here provide a related form of
geometric rigidity under perturbation, and may therefore be useful in
developing higher-dimensional analogues of weak non-degeneracy.
\section*{Acknowledgements}
\addcontentsline{toc}{section}{Acknowledgements}
The author thanks Victor Beresnevich and Sanju Velani for their invaluable guidance, insightful advice and many helpful discussions. The author is particularly grateful to Victor Beresnevich for identifying the ``chicken-and-egg'' issue during the initial development of the fibering approach, which ultimately motivated the quantitative framework presented in this paper.

\nocite{*}
\addcontentsline{toc}{section}{References}

\bibliography{Biblio}
\typeout{get arXiv to do 4 passes: Label(s) may have changed. Rerun}
\end{document}